\documentclass[a4paper,11pt]{amsart}
\usepackage[utf8]{inputenc}
\usepackage[T1]{fontenc}
\usepackage{lmodern}
\usepackage[english]{babel}
\usepackage{microtype}
\usepackage{amsmath,amssymb,amsfonts,amsthm}
\usepackage{cases}
\usepackage{mathtools,accents}
\usepackage{mathrsfs}
\usepackage{aliascnt}
\usepackage{bm}
\usepackage{bbm}
\usepackage{tikz}
\usepackage[citecolor=blue,colorlinks]{hyperref}
\usepackage{verbatim}
\usepackage{enumerate}
\usepackage{xcolor}
\usepackage{tikz-cd}
\usepackage{fullpage}
\usepackage[dvipsnames]{xcolor}

\makeatletter
\g@addto@macro\@floatboxreset\centering
\makeatother

\makeatletter
\def\newaliasedtheorem#1[#2]#3{
  \newaliascnt{#1@alt}{#2}
  \newtheorem{#1}[#1@alt]{#3}
  \expandafter\newcommand\csname #1@altname\endcsname{#3}
}
\makeatother

\newtheorem{theorem}{Theorem}[section]
\newtheorem{lemma}[theorem]{Lemma}

\newtheorem{proposition}[theorem]{Proposition}
\newtheorem{conjecture}[theorem]{Conjecture}
\theoremstyle{definition}
\newtheorem{rmk}[theorem]{Remark}
\newtheorem{remark}[theorem]{Remark}

\theoremstyle{definition}
\newtheorem{definition}[theorem]{Definition}

\newtheorem{example}[theorem]{Example}

 \newtheorem*{BHUconjI}{Riemannian Black Hole Uniqueness Conjecture}

\numberwithin{equation}{section}

\newtheorem{question}{Question}
\newtheorem{problem}[question]{Problem}
\newtheorem{conj}[question]{Conjecture}

\newcommand{\p}{\partial}
\newcommand{\bp}{\bar\partial}

\newcommand{\bl}{\bm{l}}
\newcommand{\btheta}{\bm{\vartheta}}

\newcommand{\C}{\mathbb{C}}

\newcommand{\dH}{\mathbb{H}}
\newcommand{\R}{\mathbb{R}}

\newcommand{\CP}{\mathbb{CP}}

\newcommand{\dZ}{\mathbb{Z}}
\newcommand{\dS}{\mathbb{S}}

\newcommand{\bom} {{\bm\omega}}

\newcommand{\mv} {\mathfrak v}

\newcommand{\mI} {\mathfrak I}
\newcommand{\mJ} {\mathfrak J}
\newcommand{\mR} {\mathfrak R}

\newcommand{\bH}{\mathbb H}

\DeclareMathOperator{\Aut}{Aut}

\DeclareMathOperator{\W}{W}

\DeclareMathOperator{\dvol}{dvol}

\DeclareMathOperator{\Id}{Id}

  \DeclareMathOperator{\Sc}{Scal}

\DeclareMathOperator{\Ric}{Ricci}
\DeclareMathOperator{\Rm}{Riem}

\DeclareMathOperator{\SO}{SO}

\DeclareMathOperator{\dist}{dist}

\DeclareMathOperator{\OO}{O}
\DeclareMathOperator{\I}{I}
\DeclareMathOperator{\II}{II}
\DeclareMathOperator{\III}{III}
\DeclareMathOperator{\IV}{IV}

\DeclareMathOperator{\Tr}{Tr}
\DeclareMathOperator{\tr}{tr}

\DeclareMathOperator{\Vol}{Vol}

\DeclareMathOperator{\Ae}{AE}
\DeclareMathOperator{\ALs}{AL\sharp}
\DeclareMathOperator{\GH}{GH}

\DeclareMathOperator{\ALe}{ALE}
\DeclareMathOperator{\ALF}{ALF}
\DeclareMathOperator{\ALG}{ALG}
\DeclareMathOperator{\ALH}{ALH}

\DeclareMathOperator{\EH}{EH}

\DeclareMathOperator{\ALE}{ALE}

\DeclareMathOperator{\ALFp}{ALF^+}
\DeclareMathOperator{\ALFm}{ALF^{--}}

\DeclareMathOperator{\AF}{AF}
\DeclareMathOperator{\AFa}{AF_\beta}

\usetikzlibrary{decorations.pathreplacing}

\title{Geometry of Gravitational instantons
}

\author{
Mingyang Li\textsuperscript{*}}
\thanks{\textsuperscript{*} Simons Center for Geometry and Physics, Stony Brook University; \url{mingyang.li@scgp.stonybrook.edu}}
\author{Song Sun\textsuperscript{\dag}}
\thanks{\dag\  Institute for Advanced Study in Mathematics, Zhejiang University; \url{songsun@zju.edu.cn}}

\dedicatory{Dedicated to Professor  Claude LeBrun on the occasion of his seventieth birthday}

\begin{document}
\begin{abstract}
We survey recent progress in the classification of hyperk\"ahler and Hermitian gravitational instantons (i.e., complete noncompact 4 dimensional Ricci-flat manifolds with quadratic curvature decay) , as well as the construction of non-Hermitian gravitational instantons via harmonic maps. We also present a list of open questions related to gravitational instantons.
\end{abstract}

\maketitle{}

\tableofcontents{}
\section{Introduction}
Gravitational instantons were first introduced by S. Hawking \cite{Hawking} in the study of Euclidean quantum gravity. In \cite{Hawking} they were roughly described as ``solutions of the classical Einstein equations which are non-singular on some section of complexified spacetime and in which the curvature dies away at large distances''.
In the literature, there does not seem to be a consistent definition of gravitational instantons. To fix terminology, we will use the following definition in this paper.

\begin{definition}
    A gravitational instanton is a complete, noncompact, four-dimensional oriented Riemannian manifold $(M, g)$ that satisfies the following 
    \begin{itemize}
        \item Ricci-flat: $\Ric_g\equiv 0$; 
        \item quadratic curvature decay: \begin{equation}\label{e:finite energy} |\Rm_g|=O(r^{-2}).
        \end{equation}
    \end{itemize}
\end{definition}
\begin{rmk}
\

\begin{itemize}
    \item In our definition, gravitational instantons are assumed to be \emph{oriented}. This will be convenient when discussing the \emph{Types} of gravitational instantons below.
    \item The quadratic curvature decay condition is equivalent to that $|\Rm_g|$ is square integrable, see Remark \ref{rmk-finite energy}. 
    \item In a large body of literature, it is implicitly assumed that the metric is \emph{hyperk\"ahler}, which means that the holonomy group of $g$ is contained in $SU(2)$. In our definition, we do not assume this extra property. Instead, hyperk\"ahler gravitational instantons are treated as special cases of Type $\I$ gravitational instantons. 
\end{itemize}
\end{rmk}
Gravitational instantons are  examples of Riemannian Einstein metrics. Notice that in dimensions less than 4, an Einstein metric must have constant sectional curvature, and hence is locally given by the standard space forms. This fact is related to the classical uniformization theorem in dimension 2. It also underpins the development of Ricci flow as a tool to study the geometry of 3-manifolds. In dimensions 4 and above, Einstein metrics are more flexible and they may form \emph{singularities} in families. Gravitational instantons provide local singularity models for Einstein metrics in dimension 4, thus they are expected to play an important role in the understanding of 4-dimensional geometry.

Gravitational instantons have been extensively studied in both mathematics and physics. Two fundamental questions in this area involve the \emph{construction} and \emph{classification} of gravitational instantons. These boil down to understanding the geometry of a difficult nonlinear PDE system. The classification problem can be further divided into two questions.

\

	\textbf{(A)} Classify the asymptotic structures of gravitational instantons.

    \
    
    This divides gravitational instantons into \emph{discrete} families, each characterized by distinct asymptotic behavior. An intriguing feature in the study of gravitational instantons is the existence of many different model ends, such as \(\ALE\), \(\AF\), etc.

    Given a \(\sharp\) model end \((M_0, g_0)\), we say that a gravitational instanton \((M, g)\) is \(\sharp\) if there exists a compact set \(K \subset M\) such that \(M \setminus K\) can be smoothly identified with \(M_0\), and
\[
\sup_{M \setminus K} \bigl| \nabla_{g_0}^k (g - g_0) \bigr|_{g_0} \leq C_k \, \rho^{-k-\epsilon},
\]
where \(\rho\) denotes the distance function with respect to \(g_0\).

\

	\textbf{(B)} Classify gravitational instantons with a given asymptotic structure. 
    
    \

    This involves understanding the \emph{bulk} geometry of gravitational instantons. Gravitational instantons with a given asymptotic structure can form a \emph{continuous} moduli space.

\

As we will explain in Section \ref{s-3}, we can divide gravitational instantons into three Types. 
\begin{enumerate}
	\item Type $\I$: locally (hyper-)K\"ahler ($\W^+\equiv 0$, i.e., anti-self-dual).
	\item Type $\II$: locally Hermitian but not locally K\"ahler ($\W^+$ has two distinct eigenvalues everywhere).
	\item Type $\III$: not locally Hermitian ($\W^+$ has three distinct eigenvalues generically).
\end{enumerate}
 Type $\I$ and Type $\II$ gravitational instantons are intimately related to complex geometry and can be viewed as gravitational instantons equipped with special structures. On the other hand, most known examples of gravitational instantons are all of Type $\I$ or $\II$.

\

A major recent achievement in this field is the essential classification of hyperk\"ahler and Hermitian gravitational instantons due to the work by many people in the past decade or so. These involve intriguing interactions between Riemannian geometry, complex geometry, PDEs, and also ideas from physics. On the other hand, systematic examples of Type III gravitational instanton have only been constructed very recently. In this article we will survey these  progress and discuss some open questions for future study.

\

{\bf Acknowledgments:}
The authors have benefited from discussions and collaborations with many colleagues on this topic over the past few years, especially Lars Andersson, Gao Chen, Simon Donaldson, Lorenzo Foscolo, Hans-Joachim Hein, Shaosai Huang, Marcus Khuri, Claude LeBrun, Hongyi Liu, John Lott, Jeff Viaclovsky, Gilbert Weinstein, and Ruobing Zhang.  They are also grateful for the invitation to the program ``Einstein 4-Manifolds and Gravitational Instantons'' at the Simons Center for Geometry and Physics in January 2026, where part of this article is written.

\section{General Riemannian geometric consequences}
\label{s-2}
Let $(M, g)$ be a gravitational instanton. Since it has vanishing Ricci curvature, one may deduce some basic properties of $(M, g)$ by appealing to general theory on complete Riemannian manifolds with non-negative Ricci curvature. 

Since $g$ is not flat, by the  Cheeger-Gromoll splitting theorem we know that $M$ has exactly one end.
Next, it follows from the  Bishop inequality that  for any $p\in M$, the ratio $$\nu(g; r):=\Vol(B(p,r))/\omega_4r^4$$ is a non-increasing function of $r\in (0, \infty)$, where $\omega_4$ denotes the volume of a unit ball in $\R^4$. 
The limit $$\nu_\infty:=\lim_{r\rightarrow\infty} \nu(g;r)\in [0, 1]$$ is
independent of $p$.
In particular,  for all $r>0$ we have $$\Vol(B(p,r))\leq \nu_\infty \omega_4 r^4.$$
Furthermore,   $\nu_\infty=1$ if and only if $(M, g)$ is isometric to the standard flat $\R^4$.

In general, by Gromov’s precompactness theorem, one obtains \emph{asymptotic cones} of $(M, g)$ at infinity.

\begin{definition}[Asymptotic cones]
   An asymptotic cone $(\mathcal C, O)$ is a pointed Gromov-Hausdorff limit of rescaled spaces $(M, \lambda_j^2g, p)$ along a sequence $\lambda_j\rightarrow0$.
\end{definition}
By assumption, the rescaled sequence has locally uniformly bounded curvature away from $O$. Note, however, that \emph{a priori} the limit may depend on the choice of the sequence \(\{\lambda_j\}\), so the asymptotic cone need not be unique. Moreover, a priori, an asymptotic cone is not known to be a genuine metric cone. These are the central difficulties involved in the classification of asymptotic structures of gravitational instantons; they are both confirmed in the special case of hyperk\"ahler and Hermitian gravitational instantons, see Theorem \ref{thm:SZ} and Theorem \ref{thm:Hermitian A}. 

\

We say that $(M, g)$ is \emph{non-collapsed} if $\nu_\infty>0$. This means that for a fixed $p\in M$, the unit ball around $p$ with respect to the rescaled metric $$g_{\lambda}:=\lambda^{-2}g$$ has volume bounded below by $c>0$ (i.e., the  volume  is non-collapsing) for all $\lambda\geq0$.
 In this case, it follows from the work of Bando-Kasue-Nakajima \cite{BKN} that $(M, g)$ is $\ALE$ (Asymptotically Locally Euclidean), i.e., there is a flat orbifold $(\R^4/\Gamma, g_0)$ for some $\Gamma\subset SO(4)$, whose end can be smoothly identified with $M\setminus K$ for some compact set $K$, and there is a $\delta>0$ such that for all $k\in\mathbb{Z}_{\geq0}$, we have
$$|\nabla^k_{g_0}(g-g_0)|_{g_0}\leq C_k r^{-\delta-k}$$
for some $C_k>0$. Here $r=d(p, \cdot)$ denotes the distance from $p$ on $M$, and it  is comparable to the standard radial function on $\R^4/\Gamma$. In particular, $(M, g)$ has  a \emph{unique} asymptotic cone, given by $\R^4/\Gamma$. 

At the other extreme,  a result of Calabi and Yau \cite{Calabi1970RicciII,Yau1976FunctionTheory} shows that $g$ has at least linear volume growth, that is,  there exists  $c>0$ such that $$\Vol(B(p, r))\geq cr \ \text{for all} \  r\geq 1.$$
 There is a class of gravitational instantons, namely \emph{$\ALH$} gravitational instantons (see Section \ref{ss:asymptotic models hyperkahler instantons}), which has precisely linear volume growth.


\

There are some general topological constraints on \(M\) arising from the fact that \(g\) has nonnegative Ricci curvature. A result of Anderson \cite{Anderson1990TopologyRicci} implies that
\[
b_1(M) \leq 3,
\]
and that \(M\) has a finite fundamental group if \((M, g)\) is non-collapsed. A result of Shen–Sormani \cite{ShenSormani} implies that
\[
b_3(M) = 0.
\]

 If $M$ is non-collapsed, then it has finite topological type (i.e., $M$ is diffeomorphic to the interior of a manifold with boundary), since it is $\ALE$. When $M$ is collapsed, then by \cite{CFG} the end of $M$ admits an $\mathcal N$ (i.e., nilpotent)-structure. Roughly speaking, the latter means that, locally around each point, up to a finite cover there is an effective action of a nilpotent Lie group, which satisfies suitable compatibility conditions.

\begin{lemma}\label{lem:Gauss-Bonnet}
$(M, g)$ has finite curvature energy, i.e., 
        \begin{equation}
             \int_ M |\Rm_g|^2\dvol_g<\infty.
        \end{equation} Furthermore, there is an exhaustion of $M$ by bounded open domains $\Omega_R(R\gg1)$ with smooth boundary, such that
        \begin{equation}\label{eqn:formula euler number}
     \chi(\Omega_R)=\frac{1}{8\pi^2}
        \int_ M |\Rm_g|^2\dvol_g+\delta(M),
        \end{equation}
        where $\delta(M)=\frac{1}{|\Gamma|}$ if $M$ is  $\ALe$ with asymptotic cone $\R^4/\Gamma$ and $\delta(M)=0$ if $M$ is collapsed. 
    \begin{proof}
 We will apply the Gauss-Bonnet-Chern theorem which, in our setting, asserts that for any domain $\Omega\subset M$ with smooth boundary $\p\Omega$, we have 
 \begin{equation}
 \label{e:GBC1}
\chi(\Omega) = \frac{1}{8\pi^2 }\int_{\Omega} |\Rm_h|^2\dvol_h + \int_{\p \Omega} TP_{\chi}.	
\end{equation}
The  transgression form  is given by
\begin{equation}\label{e:GBC2}
TP_{\chi} = \frac{1}{4\pi^2} \cdot \Big( -  R_{ikkj}\cdot \II_{ij} + \frac{1}{3}H^3 + \frac{2}{3}\Tr(\II^3)- H\cdot |\II|^2\Big) \dvol_{\p \Omega},	
\end{equation}
where $\II$ and $H$ denote the second fundamental form and the mean curvature of $\p \Omega$, respectively, and $i$, $j$, $k$ are in the tangential direction of $\p \Omega$.

    If $(M, g)$ is non-collapsed, then we know it is $\ALE$ by \cite{BKN}. Identifying the end of $M$ with that of $\mathbb R^4/\Gamma$ as above, for $R\gg1$ we may choose $\Omega_R\subset N$ such that $\p\Omega_R$ is $S^3_R/\Gamma\subset \mathbb R^4/\Gamma$, where $S^3_R$ denotes the standard round sphere of radius $R$. Clearly $\chi(\Omega_R)$ is independent of $R\gg1$. Applying \eqref{e:GBC1} to $\Omega_R$ and letting $R\rightarrow\infty$, we see that $g$ has finite curvature energy and the boundary term limits to $1/|\Gamma|$, \eqref{eqn:formula euler number} follows. Indeed, since $M$ has finite topological type in this case, $\chi(\Omega_R)=\chi(M)$.     
    
    If $(M, g)$ is collapsed, using the $\mathcal N$-structure, we can exhaust $M$ by bounded open domains $\Omega_R$ with $\chi(\Omega_R)$ independent of $R$ for $R\gg1$. By applying the chopping theorem of Cheeger-Gromov \cite{CheegerGromov1985Chopping},  we may also assume that $\p \Omega_R$ has uniformly bounded second fundamental form  and $\Vol(\p \Omega_R)\rightarrow 0$ as $R\rightarrow\infty$. We then apply \eqref{e:GBC1} to $\Omega_R$ and let $R\rightarrow\infty$. In this case the boundary term limits to zero, so \eqref{eqn:formula euler number} holds. 
    \end{proof}
\end{lemma}

\begin{remark}\label{rmk-finite energy}
    Although not necessary for the analysis of gravitational instantons in this article, the quadratic curvature decay assumption \eqref{e:finite energy} is indeed equivalent to that $(M, g)$ has finite curvature energy. This follows from Lemma \ref{lem:Gauss-Bonnet} and the $\epsilon$-regularity theorem of Cheeger-Tian \cite{CT}.
\end{remark}

\section{Riemannian geometry in dimension four}
\label{s-3}
Recall that the Riemann curvature tensor of a general Riemannian metric $g$ can be viewed as a self-adjoint endomorphism  of the bundle of 2-forms
$$\Rm: \Lambda^2\rightarrow \Lambda^2.$$ The standard representation theory of $SO(n)$ decomposes $\Rm$ into the sum of three components: one involving the scalar curvature $S$, one involving the trace-free Ricci curvature $\overset{\circ}{\Ric}$, and one involving the Weyl curvature $W$. In dimension four, since $SO(4)$ is a double cover of the product $SO(3)\times SO(3)$, we obtain a refined decomposition. Assuming $M$ is oriented,  the bundle $\Lambda^2$  splits as the direct sum of the bundle of self-dual and anti-self-dual 2-forms $$\Lambda^2=\Lambda^+\oplus \Lambda^-.$$
The unit sphere in $\Lambda_p^+$ can be identified with the set of orthogonal almost complex structures on $T_pM$ that are compatible with the given orientation.

Accordingly, we may write $\Rm$ as 
\begin{equation*}
\Rm=\begin{pmatrix}
  \W^++\frac{\Sc}{12} & \overset{\circ}{\Ric}\\ 
  \overset{\circ}{\Ric} & \W^-+\frac{\Sc}{12}
\end{pmatrix}.
\end{equation*} 
  At each point we may view $\W^+$ as a trace-free symmetric  $3\times 3$ matrix.  

\ 

For convenience of discussion, we divide 4-dimensional oriented Ricci-flat metrics into three classes.

\begin{enumerate}
	\item Type $\I$: $\W^+\equiv0$ (i.e.,  $g$ is anti-self-dual).  
	\item Type $\II$: $\W^+$ has exactly one eigenvalue of multiplicity two everywhere.  
	\item Type $\III$: $\W^+$ has three distinct eigenvalues generically.
\end{enumerate}

It is not completely obvious that there are exactly three Types as above. The non-trivial point is that if $\W^+$ has exactly two distinct eigenvalues on an open domain, then this must hold  everywhere. This was proved in \cite{Der}.  Using the terminology originating in the
Petrov classification of algebraic types, the Type $\I$, $\II$, $\III$ cases are also called Types $O, D, I$ respectively. Sometimes, Type $\III$ is also called \emph{algebraically general}, while Type $\I$ and $\II$ are called \emph{algebraically special}.

As we shall explain below, in the first two cases, the Ricci-flat metrics are closely related to complex geometry. This fact leads to many deep results for Type $\I$ and $\II$ gravitational instantons. At present,  we have a relatively satisfactory theory in these two cases.

\subsection{Type \texorpdfstring{$\I$}{I} case}
A Ricci-flat metric $g$ satisfies $\W^+=0$ if and only if the induced Levi-Civita connection on $\Lambda^+$ is flat. In this case, the metric $g$ has local holonomy contained in $SU(2)\subset SO(4)$. Note that under the double cover map $$SO(4)\rightarrow SO(3)\times SO(3),$$ $SU(2)$ maps onto one of the $SO(3)$ factors.  Having local holonomy contained in $SU(2)$ is the same as saying that $g$ is locally  \emph{hyperk\"ahler}. Namely, locally there is a triple of orthogonal complex structures $J_\alpha \ (\alpha=1, 2, 3)$ satisfying the quaternion relations  $$J_1J_2=-J_2J_1=J_3$$ and satisfying $$\nabla_g J_\alpha=0.$$
The set of parallel compatible complex structures is naturally parametrized by the unit 2-sphere $\mathbb S^2$: for each $\bm u=(u_1, u_2, u_3)\in \mathbb S^2$, we  associate the complex structure 
$$J_{\bm u}:=u_1J_1+u_2J_2+u_3J_3.$$
Locally, there is a triple of parallel (hence closed) 2-forms $\bom=\{\omega_{\alpha}|\alpha=1, 2, 3\}$ (where $\omega_\alpha=g(J_\alpha, \cdot)$), which gives rise to an orthonormal basis of $\Lambda^+$ at each point, i.e., 
\begin{equation}\label{eqn:hyperkahler}\frac{1}{2}\omega_\alpha\wedge \omega_\beta=\delta_{\alpha\beta} \cdot \mu
\end{equation}
where $\mu=\dvol_g$. Conversely, it is also straightfoward to see that a locally hyperk\"ahler metric $g$ is naturally a Ricci-flat metric of Type $\I$ with respect to the orientation defined by the complex structures $J_\alpha$. 

There is a characterization of the hyperk\"ahler condition in terms of a differential system involving only the triple $\bom$.  Suppose  that on a 4-manifold, we are given a triple of closed 2-forms $\bom=\{\omega_1, \omega_2,\omega_3\}$ and a volume form $\mu$ such that \eqref{eqn:hyperkahler} holds, then there is a unique hyperk\"ahler metric $g$ such that $\bom$ arises from the above construction.   Such a triple $\bom$ is usually referred to as a \emph{hyperk\"ahler triple}.

If we fix a choice of the complex structure, say $J_1$, then we can interpret the hyperk\"ahler condition in terms of a complex Monge-Amp\`ere equation. This established a connection with complex geometry. We may fix a K\"ahler form $\omega:=\omega_1$ and a holomorphic 2-form $\Omega:=\omega_2+\sqrt{-1}\omega_3$, and they satisfy the equation 
\begin{equation}\label{eqn:complex MA}\omega^2=\frac{1}{2}\Omega\wedge\bar\Omega.
\end{equation}
From this one can also see that a 4-dimensional Ricci-flat metric is locally hyperk\"ahler if and only if it is locally K\"ahler. 

In practice, it is often straightforward to write down the holomorphic 2-form  $\Omega$ using complex/algebraic geometry and obtaining a hyperk\"ahler triple reduces to finding $\omega$. Locally, $\omega$ has a K\"ahler potential $\varphi$, which solves a Monge-Amp\`ere type PDE.  On compact manifolds, the existence of K\"ahler-Einstein metrics is completely understood in terms of algebraic geometric data, due to the work of Yau \cite{Yau}, Aubin \cite{Aubin},  and Chen-Donaldson-Sun \cite{CDS}. In the noncompact setting, the connection to complex algebraic geometry is also important in the construction and classification of Type $\I$ gravitational instantons.

\subsection{Type  \texorpdfstring{$\II$}{II} case}\label{ss:Type II case}

Now suppose that $g$ is  a Ricci-flat metric of Type $\II$. Locally, we can choose a self-dual 2-form $\omega$, which is the eigen-form associated to the non-repeated eigenvalue of $\W^+$ and is normalized to have unit norm. Then $\omega$ defines an orthogonal almost complex structure $J$ such that $\omega=g(J\cdot, \cdot)$. A crucial observation by Derdzi\'nski \cite{Der} is that $J$ is \emph{integrable}. Furthermore, if we denote  $$\lambda:=2\sqrt{6}|\W^+|_g,$$ then the conformal metric
$$g_K:=\lambda^{2/3}g$$ 
is K\"ahler with respect to $J$. Since $g$ is Ricci-flat, $g_K$ has vanishing Bach tensor. In particular, $g_K$  is \emph{extremal} in the sense of Calabi as observed by LeBrun \cite{LeBrunEinsteinMetricsComplexSurfaces}, i.e., 
    $$\mathbb K:=J\nabla_{g_K}\Sc_{g_K}$$ 
is a non-zero Hamiltonian Killing field, where  the scalar curvature of $g_K$ satisfies $$|\Sc_{g_K}|=\lambda^{1/3}.$$ It then follows that $$\mathbb K=-J\nabla_g \lambda^{-1/3}$$ is also a Killing field for the original Ricci-flat metric $g$. 
In particular, a Type $\II$ Ricci-flat metric is \emph{locally} conformal to a K\"ahler metric and admits non-trivial continuous symmetries. Notice that the choice of $\omega$ is determined up to sign, so, after passing to a double cover if necessary, we may assume $g$ is \emph{globally} conformal to a  K\"ahler metric.

Conversely, note that for a K\"ahler metric in 4 dimensions, $\W^+$ always has a repeated eigenvalue $-\frac{\Sc}{12}$, and the Weyl curvature is conformally invariant. Therefore, if a Ricci-flat metric is locally conformal to a K\"ahler metric (for some complex structure), then it must be either Type $\I$ or Type $\II$. Moreover, by the Riemannian Goldberg-Sachs theorem \cite{PrzanowskiBroda1983LocallyKahler}, if a Ricci-flat metric is Hermitian with respect to an integrable almost complex structure, then it is either Type $\I$ or Type $\II$. It is Type $\I$ precisely when $W^+=0$.

\ 

For simplicity of discussion, we introduce the following terminologies in this paper 

\begin{definition}
    A 4-dimensional oriented Ricci-flat metric $g$ is called: 
    \begin{itemize}
        \item \emph{K\"ahler} if there is an orthogonal complex structure $J$ compatible with the orientation such that $\nabla_{J}g=0$;
        \item \emph{Hermitian} (or \emph{conformally K\"ahler}) if it is not K\"ahler and is conformal to a K\"ahler metric.
    \end{itemize}
\end{definition}

The above discussion then implies that being K\"ahler implies Type $\I$, and being Hermitian implies Type $\II$. Conversely, for simply-connected manifolds, Type $\I$ implies being K\"ahler, and Type $\II$ implies being Hermitian.

\section{Type \texorpdfstring{$\I$}{I} gravitational instantons}
In this section, we discuss Type $\I$ gravitational instantons. Notice that a Ricci-flat metric of Type $\I$ is in general \emph{not} globally hyperk\"ahler. It is the case when the manifold is simply-connected. 
 We  focus primarily on the case of  hyperk\"ahler gravitational instantons and  briefly discuss the general Type $\I$ case in Section \ref{ss:general Type I}.

\subsection{Gibbons-Hawking ansatz}
The ansatz of Gibbons-Hawking \cite{GH78} provides an explicit construction of 4-dimensional hyperk\"ahler metrics admitting a continuous symmetry. It also plays an important role in describing  model ends for gravitational instantons.

Let $V$ be a positive harmonic function defined on an open set $Q\subset \mathbb R^3$. If $H_2(Q; \mathbb R)=0$, then we can find a 1-form $\theta$ such that $$*dV=d\theta,$$ where $*$ denotes the Hodge star operator on $\R^3$. Consider the product space $P=Q\times \mathbb S^1$  equipped with the Riemannian metric  $g$ given by 
\begin{equation}g=V\sum_{\alpha=1}^3dx_\alpha^{2}+V^{-1}(dt+\theta)^{ 2}, 
\end{equation}
where $x_\alpha$'s are standard coordinates on $\mathbb R^3$, and $t$ denotes the standard coordinate on $\mathbb S^1=\mathbb R/\mathbb Z$.  Clearly $g$ is invariant under the obvious free $\mathbb S^1$ rotation in the $\mathbb S^1$ factor.

It is straightforward to check that $g$ is a hyperk\"ahler metric. For example, one can write down an $\mathbb S^1$-invariant triple of closed self-dual 2-forms $\bom=(\omega_1, \omega_2, \omega_3)$, by setting
\begin{align*}
   \omega_1=Vdx_1\wedge dx_2+dx_2\wedge \theta,\\
    \omega_2=Vdx_2\wedge dx_3+dx_1\wedge \theta, \\
    \omega_3=Vdx_3\wedge dx_1+dx_3\wedge \theta.
    \end{align*}

Conversely, if we are given a 4-dimensional hyperk\"ahler metric $g$ with a non-vanishing Killing field $\xi$, which is tri-holomorphic (i.e., holomorphic with respect to all the complex structures $J_\alpha$), then locally $g$ must arise from the above construction. Indeed, one can recover the coordinate $x_\alpha$ as  moment maps with respect to the  hyperk\"ahler form $\omega_\alpha$ and recover the function $V^{-\frac{1}{2}}$ as the length of $\xi$ and $\theta$ as the metric dual 1-form of $\xi$. The key point is that the presence of an $\mathbb S^1$ symmetry allows us to reduce the  nonlinear PDE governing the hyperk\"ahler structure to the Laplacian equation on $\mathbb R^3$, which is a simple \emph{linear} PDE. 

There are certain variants of the above construction which will be important later. First, we can let $Q$ be a domain in a general parallelizable flat 3 manifold and we may allow $Q$ to have non-trivial topology. The above construction still applies if we can choose a global parallel orthonormal frame of $Q$ provided the function $V$ satisfies the  integrality condition   that
\begin{equation}\label{eqn:integrality condition}\frac{1}{2\pi}\int_{C}  *dV\in \mathbb Z, \ \text{for all 2-cycles}\  C\in H_2(U;\mathbb Z).
\end{equation} 
In this case, we may replace $P$ by a principal $U(1)$ bundle $$\pi: P\rightarrow Q,$$ and replace the 1-form  $dt+\theta$  by a 1-form $\Theta$ on $P$ such that $-\sqrt{-1}\Theta$ is a connection 1-form whose curvature is $*dV$.  One can check that gauge equivalent connections lead to
isometric metrics; thus, $g$ depends only on the gauge equivalence class of $\Theta$. Different gauge equivalence classes differ by tensoring with a flat connection, and the set of isomorphism classes of
flat connections is given by $H^1(Q;\R)/H^1(Q; \mathbb Z)$.

Next, we want to include the case when the $\mathbb S^1$ action has fixed points. For this purpose we need to allow $V$ to be $+\infty$ on a discrete subset of $Q$. By B\^ocher's theorem, locally $V$ must be of the form $v+br^{-1}$ for a smooth harmonic function $v$ across the singularity and  a positive number $b$, where $r$ is the local distance function on the flat $\R^3$. Away from the singularity, for $V$ to satisfy the integrality condition \eqref{eqn:integrality condition},  $b$ must be of the form $\frac{k}{2}$ for some non-negative integer $k$. 

 In the simplest situation we can take $Q=\mathbb R^3$ and $$V=\frac{1}{2r}.$$ Then we recover the standard flat hyperk\"ahler metric on $\R^4$. Fixing a choice of a compatible complex structure we may identify $\R^4$ with $\C^2$, then the $\mathbb S^1$ action is given by \begin{equation}\label{eqn:S1 action}\lambda.(z_1, z_2)=(\lambda^{-1} z_1, \lambda z_2).
 \end{equation}The projection map $\pi:\mathbb C^2\rightarrow \mathbb R^3$ in this case is the \emph{Hopf fibration}, which can be explicitly expressed as $$(z_1, z_2)\mapsto (\text{Re}(z_1z_2), \text{Im}(z_1z_2), \frac{1}{2}(|z_1|^2-|z_2|^2)).$$ 
 If we instead take $$V=\frac{k}{2r}, \ \ k\in \mathbb Z_{>0},$$ then we get the flat orbifold $\C^2/\dZ_k$, given by the quotient of $\mathbb C^2$ by an action of $\mathbb  Z_k$ of the form $$(z_1, z_2)\mapsto (\zeta_k z_1, \zeta_k^{-1}z_2).$$
  For a general function $$V=v+\frac{k}{2r}$$ as above,  one obtains an orbifold hyperk\"ahler metric; the metric is smooth if and only if \(k = 1\).

\

Next we explain some important families of hyperk\"ahler metrics obtained using the Gibbons-Hawking ansatz.

\begin{example}[{\bf Taub-NUT metric} \cite{Taub1951, NewmanTamburinoUnti1963, Hawking}] \label{example:1}
The (positively oriented) Taub-NUT gravitational instanton can be obtained by applying the Gibbons-Hawking ansatz to \(Q = \mathbb{R}^3\), with harmonic function
\[
V = \frac{1}{2r} + 1.
\]
Since the difference from the flat model above is only the addition of a constant to \(V\), the underlying smooth manifold is again \(\mathbb{R}^4\), and the projection map \(\pi\) remains the Hopf fibration. A nontrivial fact, due to LeBrun \cite{LeBrunRicciFlatCnNotFlat}, is that for any choice of compatible complex structure, the resulting complex manifold is biholomorphic to the standard \(\mathbb{C}^2\), with the \(\mathbb{S}^1\)-action given by the standard one in \eqref{eqn:S1 action}.

However, the geometry of the Taub-NUT metric is quite different from that of the flat metric. For example, near infinity, the Hopf fibers have uniformly bounded diameter. As a consequence, the metric has \emph{cubic} volume growth. It is also easy to see that the Taub-NUT metric has a unique asymptotic cone at infinity, namely the flat \(\mathbb{R}^3\).

The Taub-NUT metric is a complete Ricci-flat metric of cohomogeneity one under the action of \(SU(2)\). In this form, it can be written explicitly as
\[
g = \frac{\rho + 1}{4\rho}\, d\rho^{2}
    + \rho(1+\rho)\, (\sigma_{1}^{2} + \sigma_{2}^{2})
    + \frac{\rho}{\rho + 1}\, \sigma_{3}^{2},
\]
where \(\sigma_1, \sigma_2, \sigma_3\) form a left-invariant coframe on \(SU(2)\), and \(\rho \in [0,\infty)\).

\end{example}

\begin{example}[{\bf Calabi-Eguchi-Hanson metric} \cite{EguchiHanson1978, Calabi1979}]\label{example:2}
    The (positively oriented) Calabi-Eguchi-Hanson gravitational instanton can be obtained by applying the Gibbons-Hawking ansatz to $Q=\mathbb R^3$, with the harmonic function
        $$V=\frac{1}{2|\bm x-\bm x_1|}+\frac{1}{2|\bm x-\bm x_2|},$$
    for $2$ distinct points $\bm x_1, \bm x_2$ in $\mathbb R^3$. The $\mathbb S^1$ action now has two fixed points. The space $P$ contains a totally geodesic embedded 2-sphere $S$ with self-intersection $-2$, which is given by the inverse image under $\pi$ of the line segment connecting $\bm x_1$ and $\bm x_2$. Indeed, $P$ can be seen to be diffeomorphic to  the total space $T^*S$, but the underlying complex manifold depends on the choice of the compatible complex structure $J_{\bm u}$. If $\bm u$ is not parallel to the segment $\bm x_1 \bm x_2$, then  $P$ is bi-holomorphic to the affine hypersurface $\{z_1^2+z_2^2+z_3^2=1\}$ in $\C^3$. If $\bm u$ is parallel to $\bm x_1\bm x_2$, $S$ becomes a complex submanifold and  $P$ is bi-holomorphic to the total space of the holomorphic line bundle $\mathcal O(-2)$ over $\mathbb C\mathbb P^1$. In the latter complex structure, the construction is a special case of a more general \emph{Calabi ansatz}. 

  The Calabi-Eguchi-Hanson metric is an \(\ALE\) gravitational instanton, with asymptotic cone \(\mathbb{C}^2/\mathbb{Z}_2\). This follows from the observation that for \(\bm{x}\) large, the function \(V\) is approximately \(\frac{2}{2r}\). 
From another point of view, if we allow the points \(\bm{x}_1\) and \(\bm{x}_2\) to vary, we obtain a family of hyperk\"ahler metrics. As \(\bm{x}_1\) and \(\bm{x}_2\) coalesce, the 2-sphere \(S\) collapses and the metric develops a singularity; the limiting space is the flat cone \(\mathbb{C}^2/\mathbb{Z}_2\).

    Similar to the Taub-NUT metric, the Calabi-Eguchi-Hanson metric is also of cohomogeneity one with respect to the group $SU(2)$. It has an explicit expression given by  \[
        g
= \frac{dr^2}{1 - \frac{a^4}{r^4}}
+ \frac{r^2}{4} \left(1 - \frac{a^4}{r^4}\right) \sigma_3^2
+ \frac{r^2}{4} \left( \sigma_1^2 + \sigma_2^2 \right),
\]
where \(r \geq a\) and $a>0$ is a constant.
\end{example}

\begin{example}\label{example:3}
 We can combine and generalize the above two examples as follows. Let \(Q = \mathbb{R}^3\), and choose \(k \geq 2\) distinct points \(\bm{x}_1, \cdots, \bm{x}_k \in \mathbb{R}^3\). For \(T \geq 0\), define
\[
V = \sum_{j=1}^{k} \frac{1}{2|\bm{x} - \bm{x}_j|} + T.
\]
Applying the Gibbons-Hawking ansatz to \(V\) yields a hyperk\"ahler gravitational instanton.

Each line segment \(\bm{x}_i \bm{x}_j\) gives rise to a 2-sphere in \(P\). Indeed, one can show that \(P\) is diffeomorphic to a smooth affine hypersurface of the form
\[
\{z_1^2 + z_2^2 + f_k(z_3) = 0\},
\]
where \(f_k\) is a generic polynomial of degree \(k\). The underlying complex manifold is more subtle to describe. For a complex structure \(J_{\bm{u}}\), whenever the segment \(\bm{x}_i \bm{x}_j\) is parallel to \(\bm{u}\), the corresponding 2-sphere becomes holomorphic with respect to \(J_{\bm{u}}\).

The asymptotic geometry of the resulting hyperk\"ahler metric depends on \(T\). If \(T = 0\), then for \(\bm{x}\) large, \(V\) is approximately \(\frac{k}{2r}\), so the metric is \(\ALE\) with asymptotic cone \(\mathbb{C}^2/\mathbb{Z}_k\). 
If \(T > 0\), then for large \(r\), \(V\) is approximately \(\frac{k}{2r} + T\), so the \(\mathbb{S}^1\)-fiber has bounded length and the metric is asymptotic to a \(\mathbb{Z}_k\)-quotient of the Taub-NUT metric. In this case, the metric has cubic volume growth and a unique asymptotic cone given by \(\mathbb{R}^3\). These are called the \emph{multi-Taub-NUT metrics}.
\end{example}


\subsection{Asymptotic models}\label{ss:asymptotic models hyperkahler instantons}
We now introduce several families of asymptotic model ends for gravitational instantons. Throughout this section, unless otherwise specified, we focus on the end at infinity of the spaces under consideration.

We begin with the four standard model ends: \(\ALE\), \(\ALF\), \(\ALG\), and \(\ALH\), which are asymptotically locally products of the form \(\mathbb{T}^k \times \mathbb{R}^{4-k}\) for \(k = 0,1,2,3\), respectively. The terminology \(\ALF\) originally stands for ``Asymptotically Locally Flat,'' whereas \(\ALG\) and \(\ALH\) are not derived from abbreviations.

\subsubsection{$\ALE$ models}
For a finite subgroup \(\Gamma \subset SO(4)\) acting freely on \(S^3\), an \(\ALE\) (asymptotically locally Euclidean)-\(\Gamma\) model end is given by the flat cone \(\mathbb{R}^4/\Gamma\). Such finite subgroups \(\Gamma\) have been completely classified. 

After identifying \(\mathbb{R}^4 \simeq \mathbb{C}^2\), we may assume that \(\Gamma\) is a subgroup of \(U(2) \subset SO(4)\). In particular, an \(\ALE\) model end is always K\"ahler with respect to some choice of complex structure. It is hyperk\"ahler if, in addition, \(\Gamma\) can be chosen to lie in \(SU(2)\). In this case, \(\Gamma\) is one of the \emph{ADE} groups, which are classified as follows:
\begin{itemize}
    \item $A_k(k\geq 0)$: this is the cyclic group of order $k+1$, generated by the diagonal matrix 
     \[
\begin{pmatrix}
\zeta_{k+1}& 0 \\
0 & \zeta_{k+1}^{-1}
\end{pmatrix}.
\]
    \item $D_k(k\geq 3)$: this is the binary dihedral group generated by 
    \[
\begin{pmatrix}
\zeta_{k-2}& 0 \\
0 & \zeta_{k-2}^{-1}
\end{pmatrix},  \ \ \begin{pmatrix}
0 & 1 \\
-1 & 0
\end{pmatrix}.
\]
\item $E_k(k=6,7, 8)$: these correspond to the binary tetrahedral, octahedral, icosahedral groups.
\end{itemize}

The complete classification of finite subgroups of $U(2)$ acting freely on $S^3$ can be found in \cite[Theorem 4.11]{Scott1983Geometries}.

Notice that the $\ALE$-$A_k$ model end admits an $\mathbb S^1$ symmetry, so it is given by the Gibbons-Hawking ansatz. Indeed, one can take $Q=\R^3\setminus B_1$ and $$V=\frac{k+1}{2r}.$$

\subsubsection{$\ALF$ models}

An $\ALF$-$A_k\ (k\in \dZ)$ model end  is given by  the Gibbons-Hawking ansatz applied to $Q=\mathbb R^3$ and the harmonic function $$V=\frac{k+1}{2r}+T$$
for some $T>0$. In particular, it is hyperk\"ahler.  When $k\geq 0$, it is given by the $\mathbb Z_{k+1}$ quotient of the Taub-NUT metric.

An $\ALF$-$D_k(k\in \dZ)$ model end is the quotient of an $\ALF$-$A_{2k-5}$ end by a further $\mathbb Z_2$ action. The $\mathbb Z_2$ action can be described in the Gibbons-Hawking ansatz as an involution $\bm x\mapsto -\bm x$ on $\R^3$ and $t\mapsto -t$. An $\ALF$-$D_k$ model end is also hyperk\"ahler.

A general $\ALF$ model end is a finite free isometric quotient of an $\ALF$-$A_k$ end for some $k$. It is not necessarily hyperk\"ahler. 

Note that an \(\ALF\) model end has cubic volume growth. At infinity, it is a locally asymptotically flat \(\mathbb{S}^1\)-fibration. The asymptotic cone is \(\mathbb{R}^3\) in the \(A_k\) case, \(\mathbb{R}^3/\mathbb{Z}_2\) in the \(D_k\) case, and more generally an orbifold of the form \(\mathbb{R}^3/\Gamma\).

\subsubsection{$\ALG$ models} 

For $\beta\in (0, 1]$, take a 2-dimensional flat cone $\mathbb C_\beta$ with angle $2\pi\beta$, the cotangent bundle $T^*\mathbb C_\beta$ admits a canonical flat metric. Indeed, the metric is hyperk\"ahler. Take a lattice sub-bundle, we obtain a flat end, which we denote by $\ALG_\beta$.

An $\ALG_\beta$ model end is always K\"ahler. It  is hyperk\"ahler exactly when the lattice sub-bundle is preserved by the monodromy $M_\beta$, which is the case exactly when $\beta\in \{\frac12, \frac13, \frac23, \frac14, \frac34, \frac16, \frac56, 1\}$ (see for example \cite[Section 3.2]{SZ}).

An $\ALG$ model end has quadratic volume growth. At infinity it is a locally asymptotically flat $T^2$ fibration, and the asymptotic cone is exactly $\mathbb C_\beta$.

\subsubsection{$\ALH$ models}
An $\ALH$ model end is given by  $[0, \infty)\times F$ for a compact flat 3-manifold $F$. It is  hyperk\"ahler if $F$ is a flat torus. 

An $\ALH$ model end has linear volume growth, with asymptotic cone given by the ray $[0, \infty)$.

\

Next we introduce two two exceptional families of hyperk\"ahler model ends. 
They are given by the Gibbons-Hawking ansatz applied to slightly more general flat 3-manifolds. 

\subsubsection{$\ALH^*$ models}

Consider the product $$Q=\mathbb T\times [1, \infty)$$ for an oriented flat 2-torus $\mathbb T=\R^2/\Lambda$. Denote by $g_{\mathbb T}+dz^2$ the product metric  on $Q$ and fix the orthonormal frame $e_1=\p_x, e_2=\p_y, e_3=\p_z$, where $\{x, y\}$ are the standard coordinates on $\R^2$. Denote by $A$ the area of $\mathbb T$. Consider the harmonic function $$V=\frac{2\pi kz}{A}$$ for $k\in \mathbb Z_{>0}$. 
The integrality condition \eqref{eqn:integrality condition} is satisfied, so  we can apply the Gibbons-Hawking ansatz here.  Fix a choice of $\Theta$, then we obtain   an  $\ALH^*_k$ model end. As is explained in \cite[Section 2]{HSVZ22}, for a different choice of $\Theta$, the hyperk\"ahler metric only differs by a diffeomorphism. 

Since $V$ is itself invariant under the translation along $\mathbb T$, $g$ admits extra symmetries. Indeed, there is a tri-holomorphic isometric action of the 3-dimensional Heisenberg group on $(P, g)$, and the action is transitive on the level set of $z$.

An $\ALH^*$ model end exhibits interesting \emph{multi-scale} Riemannian geometry: as $z$ tends to infinity, the size of the $\mathbb S^1$ orbits shrinks, whereas the size of the base $\mathbb T^2$ expands. See Figure \ref{f:Calabi model}. It is easy to see that the volume growth rate is given by $\frac{4}{3}$, and the asymptotic cone is the ray $[0, \infty)$.

\begin{figure}[h]
\begin{center}
\begin{tikzpicture}[scale = 0.8]
\draw (2,.5) to [out = 0, in = 235] (4,2);
\draw (4, 2) to [out=55, in=255] (4.5,3);
\draw (4, -2) to [out=-55, in=-255] (4.5,-3);
\draw (2,-.5) to [out = 0, in = 125] (4,-2);
\draw[red, fill=pink](2,0) ellipse (.2 and .5);
\draw[gray](2,0) ellipse (.2 and .05);
\draw[red, fill=pink] (3,0) ellipse (.1 and .82);
\draw[gray](3,0) ellipse (.1 and .05);
\draw[red, fill=pink](4,0) ellipse (.05 and 2);
\draw[gray](4,0) ellipse (.05 and .02);
\node[align = left, below] at (7, 0.5) {\footnotesize{$\mathbb S^1$ bundle over $\mathbb T^2$}};
\draw[dashed, ->] (5.4, 0.1) -- (4.2, 0.); 

\end{tikzpicture}
\end{center}
\caption{The Calabi model space}	
\label{f:Calabi model}
\end{figure}
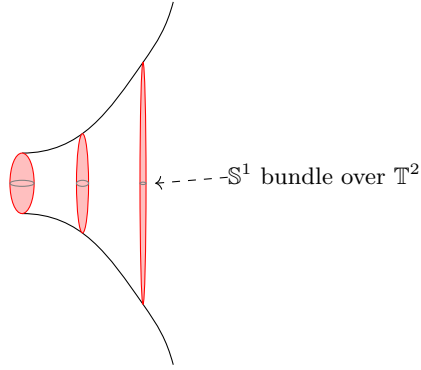

The \(\ALH^*\) model ends admit several interesting descriptions from the viewpoint of complex geometry. On the one hand, one may choose a compatible complex structure on \((P,g)\) such that the natural projection \(P \to \mathbb{T}\) is holomorphic. In this case, \(P\) is biholomorphic to the complement of the zero section in the disk bundle inside the total space of a holomorphic line bundle \(L\) of degree \(k\) over \(\mathbb{T}\). Thus the complete end of \(P\) can be compactified complex-analytically by adding the zero section of \(L\) (the holomorphic structure on \(L\) depends on the choice of the connection \(\Theta\)).

From this perspective, the construction fits into a special case of the \emph{Calabi ansatz}, which produces special K\"ahler metrics on a Hermitian line bundle \((L,h)\) over a compact K\"ahler manifold, of the form
\[
\sqrt{-1}\,\partial\bar{\partial} F(|\xi|_h^2).
\]
The general Calabi ansatz extends to higher dimensions; see, for example \cite{BandoKobayashi1990, TianYau1990, HSVZ2, chen2024calabiyaumetricscalabitype}.

On the other hand,  we can choose another compatible complex structure on $(P, g)$ so that the natural projection map $P\rightarrow \mathbb S^1\times (1, \infty)$ is holomorphic, for some totally geodesic circle $\mathbb S^1\subset \mathbb T$. Notice that $\mathbb S^1\times [1, \infty)$ can be holomorphically identified with a punctured disk $\Delta^*$. Then $P$ becomes an elliptic fibration over $\Delta^*$, with monodromy of type $I_k$. From standard theory the complete end of $P$ can be compactified into a smooth complex surface which is an elliptic fibration over $\Delta$, with a singular fiber at $0$, of Type $I_k$ in terms of the Kodaira classification. 

\subsubsection{$\ALG^*$ models}
Similarly we consider the flat product $$Q=\mathbb S^1\times (\mathbb R^2\setminus B_1), $$
where $B_1$ denotes the standard unit ball in $\R^2$. Denote by $\ell$ the length of $\mathbb S^1$. Take the harmonic function
$$V=\frac{k\log r}{\ell}$$
for $k\in\mathbb Z_{>0}$, where $r$ is the standard radial function on $\mathbb R^2$. Again by fixing a $U(1)$ connection $-\sqrt{-1}\Theta$ the Gibbons-Hawking construction gives rise to a $U(1)$ invariant hyperk\"ahler metric $g$ on a $U(1)$ bundle $P\rightarrow Q$ of degree $k$. This is what we call an $\ALG^*$-$I_k$ model end.  

An $\ALG^*$-$I_k^*$ model end is a free $\mathbb Z_2$ quotient of an $\ALG^*$-$I_k$ model end, where the $\mathbb Z_2$ action covers the standard involution on $\R^2$ and on $\mathbb S^1$. 

As above, an \(\ALG^*\)-\(I_k\) model end admits a tri-holomorphic isometric action of the 3-dimensional Heisenberg group. On the \(\mathbb{Z}_2\)-quotient, however, only a 2-dimensional torus action descends. The level sets of \(r\) are \emph{infranil manifolds}, namely circle bundles over a Klein bottle. The geometry also exhibits a multi-scale structure: when viewed as a torus fibration, the fibers degenerate as \(r \to \infty\).

One can check that the volume growth is quadratic. The asymptotic cone is \(\mathbb{R}^2\) in the \(\ALG^*\)-\(I_k\) case, and \(\mathbb{R}^2/\mathbb{Z}_2 = \mathbb{C}_\pi\) in the \(\ALG^*\)-\(I_k^*\) case.

From the viewpoint of complex geometry, one may choose a compatible complex structure on \(P\) such that the projection map \(P \to \mathbb{R}^2 \setminus B_1\) is holomorphic. Identifying \(\mathbb{R}^2 \setminus B_1\) holomorphically with the punctured disk \(\Delta^*\), the space \(P\) becomes an elliptic fibration over \(\Delta^*\), with monodromy of type \(I_k\) or \(I_k^*\). Moreover, it admits a compactification to an elliptic fibration over \(\Delta\).

\begin{remark}\label{rem:curvature decay rate}
  Among the asymptotic model ends described above, the \(\ALE\), \(\ALG\), and \(\ALH\) ends are flat, while the \(\ALF\) ends have cubic curvature decay. The \(\ALH^*\) ends exhibit precisely quadratic curvature decay, and the \(\ALG^*\) ends have curvature decay of order \(O\!\left(r^{-2}(\log r)^{-1}\right)\).
\end{remark}

\subsection{Constructions of examples}
There are diverse constructions of examples of hyperk\"ahler gravitational instantons in both the mathematics and physics literature. 

$\ALE$-$A_k$$(k\geq 1)$ hyperk\"ahler gravitational instantons can be constructed  using the Gibbons-Hawking ansatz, as given in Example \ref{example:3}. General $\ALE$ hyperk\"ahler gravitational instantons were constructed by Kronheimer \cite{Kronheimer89C} using the technique of hyperk\"ahler quotients (see \cite{HKLR87}).

 $\ALF$-$A_k$$(k\geq 1)$ hyperk\"ahler gravitational instantons can  be constructed using the Gibbons-Hawking ansatz, as given in Example \ref{example:3}. 
$\ALF$-$D_k$ $(k\geq 0)$ gravitational instantons can be constructed by the twistor theory and generalized legendre tranform (see for example \cite{Hit84, LR88, IR96, CK98, CK99, CH05}) and the  gluing method (see for example) \cite{BM11, Auv18, SS21, Zhu24}. $\ALF$-$D_{0}$ gravitational instanton (up to scaling) is famously known as the \emph{Atiyah-Hitchin metric}, which was originally constructed as the moduli space of centered charge two $SU(2)$ monopoles on $\R^3$ (see \cite{AH88}). It has fundamental group $\mathbb Z_2$, and its double cover is an $\ALF$-$D_1$ gravitational instanton; general $\ALF$-$D_1$ gravitational instantons were constructed by Dancer \cite{Dancer1993} using Nahm's equation and hyperk\"ahler quotients. $\ALF$-$D_2$ case was first speculated by Page \cite{Page1979} as a \emph{periodic but non-stationary  gravitational instanton}, and this was rigorously constructed by Hitchin \cite{Hit84} using the twistor theory.

Examples of $\ALG, \ALH, \ALG^*, \ALH^*$ hyperk\"ahler gravitational instantons can be constructed using solutions to the complex Monge-Amp\`ere equation  by the work of Tian-Yau-Hein \cite{TianYau1990, Hein}. These involve abstract PDE method and the metrics do not have explicit formulae in general.  In \cite{Hein} these were produced on the complement of anti-canonical fiber $D$ in a rational elliptic surface $X$. 
More specifically, in terms of the Kodaira classification of singular fibers, the metrics are 
\begin{itemize}
    \item $\ALH$ if $D$ is smooth.
    \item $\ALG$ if $D$ is singular with finite monodromy. Here the correspondence between the parameter $\beta$ and the Kodaira type is given as follows
   $$ \begin{tabular}{cccccccc}
Kodaira Type & $\I_0^*$ & $\II$ & $\II^*$ & $\III$ & $\III^*$ & $\IV$ & $\IV^*$ \\

$\beta$ & $\frac12$ & $\frac16$ & $\frac56$ & $\frac14$ & $\frac34$ & $\frac13$ & $\frac23$ \\
\end{tabular}.$$
    \item $\ALH_k^*$ if $D$ singular with infinite monodromy of Kodaira type $I_k$. Here the parameter $k\in \{1, \cdots, 9\}$.
    \item $\ALG^*$-$I_k^*$ if $D$ is singular with infinite monodromy of Kodaira type $I_k^*$. Here the parameter $k\in \{1, 2, 3, 4\}$.
\end{itemize}
The strategy of the construction goes as follows. First it is easy to write down an explicit holomorphic volume form $\Omega$ on $X\setminus D$. Then one can write down a background K\"ahler metric $\tilde\omega$ on $X\setminus D$ which agrees with an explicit \emph{semi-flat} hyperk\"ahler metric near $D$ and which has the asymptotics determined by the Kodaira type of $D$. Finally one solves the complex Monge-Amp\`ere equation \eqref{eqn:complex MA} using  $\tilde\omega$ as a background metric, i.e., by writing $\omega=\tilde\omega+\sqrt{-1}\p\bp\phi$ and solving the equation for $\phi.$

Special examples of these can be constructed using the gluing technique (see for example \cite{BM11, Auv18}). Examples of $\ALH^*$ hyperk\"ahler gravitational instantons can also be constructed on the complement of an anti-canonical divisor $D$ in a del Pezzo surface $X$ (see \cite{TianYau1990, HSVZ2}), where the background metric is a constructed using the Calabi ansatz, noticing that a neighborhood of $D$ is modeled on the total space of a positive line bundle over $D$. 

There are also many works in the literature which aim to construct hyperk\"ahler gravitational instantons (and hyperk\"ahler metrics in general dimensions too) using gauge theory. These include for example the moduli space of singular monopoles and the Hitchin moduli space of Higgs bundles, see for example \cite{Hitchin1987SelfDuality, Hitchin1987Stable, AH88,Nakajima1996, CK98, CK99, CherkisKapustin2001, Hausel2001, CherkisKapustin2002, CherkisKapustin2003, BiquardBoalch2004, Boalch2012Hyperkahler, GMN2013, Neitzke2014HyperkahlerNotes, Foscolo2016, Foscolo2017Gluing, Boalch2018Wild, MazzeoSwobodaWeissWitt2019, Fredrickson2020, ChenLi, TripathyZimet2203.13730,  Mochizuki2023HitchinAsymptotics,HHL} and the references therein.

\subsection{Classification Results A: Asymptotic Structure}
From the above discussion, we see that there is a wide variety of asymptotic models for hyperk\"ahler gravitational instantons. Nevertheless, a complete classification of their asymptotic ends is now available. Notice that a flat hyperk\"ahler gravitational instanton must be a product $\R^k\times \mathbb T^{4-k}$, so we may focus on the non-flat case. 

\begin{theorem}[Asymptotic classification \cite{SZ}] \label{thm:SZ}
Let $(M, g)$ be a hyperk\"ahler gravitational instanton which is not flat, then it must be $\ALE, \ALF, \ALG, \ALH, \ALG^*$ or  $\ALH^*$.
\end{theorem}

In the non-collapsed case it follows from the classical result of Bando-Kasue-Nakajima \cite{BKN} that $(M, g)$ is $\ALE$.  The proof of Theorem \ref{thm:SZ} is mainly concerned with the collapsed case. It relies on the interaction between the general collapsing theory of Cheeger-Fukaya-Gromov \cite{CFG} in Riemannian geometry and the analysis of the hyperk\"ahler equation \eqref{eqn:hyperkahler}. Below we describe some key steps in the argument.

We always assume $(M, g)$ is collapsed. 
The first major step in the proof is to classify the \emph{asymptotic cones} of $(M, g)$. Recall that an asymptotic cone $(\mathcal C, O)$ is a pointed Gromov-Hausdorff limit of $(M_j, g_j, p_j):=(M, \lambda_j^2g, p)$ for a sequence $\lambda_j\rightarrow0$. A priori, $\mathcal C$ may depend on the choice of the sequence $\{\lambda_j\}$, and it may not even be a genuine metric cone. These were only known \cite{Kasue,MNO} to hold under the assumption that $(M,g)$ has faster-than-quadratic curvature decay, using general Riemannian geometric arguments. The following result establishes these facts for hyperk\"ahler gravitational instantons. 

\begin{theorem}[Classification of Asymptotic Cones, \cite{SZ}]\label{thm:unique tangent cone at infinity}
    There is a unique asymptotic cone of $(M, g)$, which is a flat metric cone of one of the following forms
    $$\R^3, \R^3/\mathbb Z_2, \mathbb C_\beta(\beta\in \{1,\frac12, \frac13, \frac23, \frac14, \frac34, \frac16, \frac56\}), \mathbb R, \mathbb R_{\geq 0}$$
\end{theorem}

We give a sketch of the proof of Theorem \ref{thm:unique tangent cone at infinity} here. Let $(\mathcal C, O)$ be an asymptotic cone. By the quadratic curvature decay assumption,  we know  that locally away from $O$, the convergence to $\mathcal C$ is collapsing with uniformly bounded curvature. General Cheeger-Fukaya-Gromov theory implies that $\mathcal C^\circ:=\mathcal C\setminus \{0\}$ is locally given by the quotient of a smooth Riemannian manifold by a compact Lie group and the collapsing is locally along an infranil fibration, possibly with singular fibers. In our special setting of 4-dimensional hyperk\"ahler metrics one can show that $\mathcal C^\circ$ is indeed a smooth Riemannian manifold, and the collapsing is locally along a nilpotent fibration. In particular, the fibers are either tori or three dimensional Heisenberg manifolds.

The hyperk\"ahler structure on $M$ gives a special \emph{affine structure} on $\mathcal C^\circ$ (the precise meaning depends on the dimension of $\mathcal C^\circ$). Take any $q\in \mathcal C^\circ$ and a sequence $q_j\in (M_j, g_j)$ converging to $q$. One can find $r>0$ such that the universal cover $\widetilde B_j$ of $B(q_j, r)$ is locally volume non-collapsing, and after passing to subsequences it converges to a  4-dimensional smooth limit manifold $\widetilde B_\infty$ equipped with a hyperk\"ahler metric. Moreover, the action of the deck transformation group $G_j$ on $\widetilde B_j$ converges to the tri-holomorphic isometric action of a Lie group $G_\infty$ of positive rank, and $B(q, r)$ is  given by the quotient $\widetilde B_\infty/G_\infty$.
To understand the global geometry of $\mathcal C$ one needs to  analyze the singularity behavior at $O$. For general Einstein equation, this is a challenging question in the volume collapsing situation. Notice that there is no a priori knowledge about the topology of $\mathcal C$ around $O$. The main observation in \cite{SZ} is that, in our hyperk\"ahler setting, the singularity at $O$ can be controlled by the special affine structure in a punctured neighborhood of $O$. 
 
Let us focus on the case when the dimension of $\mathcal C^\circ$ is  3. In this case $G_\infty$ is generated by a single tri-holomorphic Killing field. As a converse to the Gibbons-Hawking ansatz, the limit Riemannian metric $g_\infty$ on $\mathcal C^\circ$ is locally of the form 
$$g_\infty=V(dx^2+dy^2+dz^2),$$
where the coordinates $x, y, z$ can be chosen to be globally well-defined up to a translation and a simultaneous rescaling by a constant $C\neq 0$,
and $V>0$ is a harmonic function with respect to the flat metric $$g^\flat:=V^{-1}g_\infty.$$ 

The analysis essentially reduces to the study of a positive harmonic function on an (incomplete) flat 3-manifold. A crucial step is to show that, near \(O\), the metric \(g^\flat\) can be completed by adding exactly one point. This requires a growth estimate for \(V^{-1}\) near the singularity.

To obtain such an estimate, observe that if \(V^{-1}\) were to grow too quickly, then \(V\) would decay correspondingly fast and hence extend across the singularity as a harmonic function in the sense of metric measure spaces. The maximum principle would then imply that \(V\) is \emph{bounded} near the singularity, leading to a contradiction.

If one knew that the metric completion of \(g^\flat\) were a flat 3-dimensional cone (i.e., \(\mathbb{R}^3\) or \(\mathbb{R}^3/\mathbb{Z}_2\)), then B\^{o}cher's theorem would imply that \(V\) has the form
\begin{equation}\label{eqn:harmonic function local form}
V = \frac{c}{r} + v,
\end{equation}
for some \(c > 0\) and a smooth harmonic function \(v\) extending across the singularity. In practice, however, such a removable singularity theorem is not available (see Question~4.13 in \cite{SZ}).

Instead, \cite{SZ} exploits the additional structure coming from the special affine coordinates \(\{x,y,z\}\), whose monodromy lies in \(\mathbb{R}^3 \rtimes \mathbb{Z}_2\), together with the classification of ends of flat manifolds by Eschenburg-Schroeder \cite{ES}.

In our setting, one can further show that the constant \(c\) must vanish. This uses the Gauss-Bonnet-Chern theorem, the finite curvature energy of \((M,g)\), and a result on topological lower bounds for symplectic fillings of spherical 3-manifolds \cite{OhtaOno2005}. As a consequence, the singularity is a 3-dimensional smooth orbifold with a unique tangent cone given by flat \(\mathbb{R}^3\) or \(\mathbb{R}^3/\mathbb{Z}_2\).

A similar analysis applies when \(\mathcal{C}^\circ\) has dimension \(2\), where the singularity of \(\mathcal{C}\) is described by a singular \emph{special K\"ahler structure}, and the possible monodromies can be classified.

Using the local analysis together with the completeness of \(\mathcal{C}\), one can further show that, in general, \(\mathcal{C}\) must be isometric to \(\mathbb{R}^k \times \mathbb{T}^{3-k}\) for \(k = 1,2,3\), or to a \(\mathbb{Z}_2\)-quotient thereof. By the rigidity of these spaces, it follows that there is a unique asymptotic cone, which is a metric cone belonging to the list in Theorem~\ref{thm:unique tangent cone at infinity}.

\

The second major step in proving Theorem~\ref{thm:SZ} is to \emph{create} local symmetries on the end of \(M\). The Cheeger-Fukaya-Gromov theory yields an \(\mathcal{N}\)-structure on the end. In our setting, this means, roughly speaking, that there exists a local free action of a nilpotent group \(\mathcal{N}\) on \(M \setminus K\), for some compact set \(K \subset M\), such that under rescaling to the asymptotic cone \(\mathcal{C}\), the \(\mathcal{N}\)-orbits correspond precisely to the collapsing fibers.

The following result reduces the analysis of the asymptotic behavior of \(g\) to the \(\mathcal{N}\)-invariant case.

\begin{theorem}[Perturbation Theorem \cite{SZ}]\label{thm:perturbation}
    By enlarging $K$ if necessary, there is an $\mathcal N$-invariant hyperk\"ahler metric $\widetilde g$ on $M\setminus K$ and $\epsilon>0$ such that for all $k\geq0$,
    \begin{equation}\label{eqn:perturbation}
        \|\nabla_g^k(\widetilde g-g)\|_{g}\leq C_k r^{-\epsilon-k}.
    \end{equation}
\end{theorem}

The proof of Theorem~\ref{thm:perturbation} involves two main ingredients. First, it was shown in \cite{CFG} that by averaging over the \(\mathcal{N}\)-orbits, one can construct an \(\mathcal{N}\)-invariant Riemannian metric \(\widetilde{g}\) satisfying \eqref{eqn:perturbation}. However, in general, this averaging procedure destroys the special structures associated with \(g\).

Secondly, a quantitative implicit function theorem is established in \cite{SZ} to further perturb \(\widetilde{g}\) to a new metric that is hyperk\"ahler. This step relies crucially on Theorem~\ref{thm:unique tangent cone at infinity} and is inspired by related gluing constructions for hyperk\"ahler metrics on K3 surfaces (see \cite{Foscolo2016, HSVZ22}).

\

To complete the proof of Theorem~\ref{thm:SZ}, it remains to classify \(\mathcal{N}\)-invariant hyperk\"ahler ends. We illustrate this in the case where the asymptotic cone \(\mathcal{C}\) is \(\mathbb{R}^3\). In this situation, the \(\mathcal{N}\)-structure is given by an \(\mathbb{S}^1\)-fibration, and the metric \(\widetilde{g}\) is locally described by the Gibbons-Hawking ansatz.

As explained above, the orbit space \(\mathcal{Q}\) of the local \(\mathbb{S}^1\)-action inherits a special affine structure. In this setting, one can show that \(\mathcal{Q}\) may be identified with the complement of a compact set in \(\mathbb{R}^3\). One consequence of Theorem~\ref{thm:perturbation} is that the length of the \(\mathbb{S}^1\)-orbits, given by \(V^{-1/2}\), satisfies
\begin{equation}
r^{-\delta} \leq V^{-1/2} \leq r^{\delta}
\end{equation}
for all \(\delta > 0\). It follows that \(V\) must have the form
\[
V = c + \frac{k+1}{2r} + O(r^{-2}),
\]
for some constants \(c \in \mathbb{R}\) and \(k \in \mathbb{Z}\). Consequently, \(\widetilde{g}\), and hence \(g\), is asymptotic to the \(\ALF\)-\(A_k\) model end.

\

We remark that a partial classification of gravitational instantons was obtained by Chen-Chen \cite{Chen-Chen1} under the additional assumption that the curvature decays at the rate \(O(r^{-2-\epsilon})\) for some \(\epsilon > 0\). In this case, the metric is locally asymptotic to a flat torus bundle, and only the four families \(\ALE\), \(\ALF\), \(\ALG\), and \(\ALH\) can occur.

The arguments in \cite{Chen-Chen1} are based on the approach of Minerbe \cite{Minerbe}, which studies the holonomy of short geodesic loops at infinity via \emph{ODE} comparison. By contrast, the approach developed in \cite{SZ} is quite different: it makes use of the full strength of Cheeger-Fukaya-Gromov theory. This has the advantage of, on the one hand, accommodating possible \emph{multi-scale} collapsing phenomena at infinity, and on the other hand, highlighting the role of the hyperk\"ahler equation as an elliptic \emph{PDE}.

\

\subsection{Classification Results B: Global Structure}
To classify hyperk\"ahler gravitational instantons, Theorem~\ref{thm:SZ} reduces the problem to understanding all hyperk\"ahler gravitational instantons \((M,g)\) with a prescribed asymptotic model. There are several aspects to consider:
\begin{itemize}
\item (Complex geometry) For a choice of compatible complex structure \(J\), describe the metric \(g\) in terms of the complex-geometric data of \((M,J)\).
\item (Torelli problem) Characterize the metric \(g\) in terms of cohomological data (i.e., periods).
\end{itemize}


\subsubsection{Complex geometry}
To understand the complex geometry of gravitational instantons, we recall the longstanding compactification conjecture of Yau \cite{Yau1982problem}.

\begin{conjecture}\label{conj:Yau}
A complete Ricci-flat K\"ahler manifold can be compactified into a compact complex manifold by adding a divisor at infinity.
\end{conjecture}

This conjecture is known to fail in general; a counterexample is given by the 4-dimensional complete hyperk\"ahler metrics constructed in \cite{AKL}. The following result confirms a version of Conjecture~\ref{conj:Yau} in complex dimension \(2\).

\begin{theorem}[Compactification Theorem] \label{thm:compactification}
    Given a hyperk\"ahler gravitational instanton $(M, g)$, there is a compatible complex structure $J$ such that $(M, J)$ is bi-holomorphic to the complement of a divisor $D$ in an algebraic surface $X$. 
\end{theorem}
 Theorem~\ref{thm:compactification} follows from Theorem~\ref{thm:SZ}, together with several works that treat each asymptotics separately, which we will describe below. Given Theorem~\ref{thm:compactification}, a natural question is to further classify the hyperk\"ahler metric in terms of more refined complex-geometric data; this will also be discussed below.

In the \(\ALE\) case, by the classical work of Kronheimer \cite{Kronheimer89T}, we know that for \emph{any} choice of compatible complex structure \(J\), \((M,J)\) is given by the minimal resolution of a partial smoothing of a complex singularity \(\mathbb{C}^2/\Gamma \subset \mathbb{C}^3\) for some \(\Gamma \subset SU(2)\). These spaces can be compactified into a projective orbifold by adding a natural orbifold divisor at infinity.
The metric \(g\) is, up to scaling, characterized by the choice of a K\"ahler class on \((M,J)\). More generally, it belongs to the class of asymptotically conical Ricci-flat K\"ahler metrics in arbitrary dimension; see \cite{BandoKobayashi1990, TianYau1991, ConlonHein2015ACCY}.

In the \(\ALF\)-\(A_k\) case, it was proved by Minerbe \cite{Minerbe2009Rigidity} that all such metrics arise from the Gibbons-Hawking ansatz, and are precisely the multi-Taub-NUT metrics. For any choice of compatible complex structure \(J\), \((M,J)\) is given by the minimal resolution of a partial smoothing of an \(A_k\) singularity. In particular, one must have \(k \geq 1\).

In the general \(\ALF\) case, it was proved by Chen-Chen \cite{Chen-Chen2} that for any choice of \(J\), \((M,J)\) admits a compactification to a rational ruled surface. In the \(\ALF\)-\(D_k\) case, the hyperk\"ahler metrics can be classified in terms of their twistor spaces, and are precisely those constructed by Cherkis-Hitchin-Ivanov-Kapustin-Lindstr\"om-Ro\v{c}ek \cite{Hit84, LR88, IR96, CK98, CK99, CH05}. In particular, one must have \(k \geq 0\).

For the remaining four families, it is known that for \emph{some} choice of compatible complex structure \(J\), \((M,J)\) can be compactified into a rational elliptic surface by adding an anti-canonical fiber. This is proved by Chen-Chen \cite{Chen-Chen1, Chen-Chen3} in the \(\ALG\) and \(\ALH\) cases, by \cite{ChenViaclovsky2022QuadraticVolume} in the \(\ALG^*\) case, and by \cite{CollinsJacobLin2021SLagTori, CJL1, HSVZ2} in the \(\ALH^*\) case. In particular, in the \(\ALG\) case one must have \(\beta \in \{\tfrac12, \tfrac13, \tfrac23, \tfrac14, \tfrac34, \tfrac16, \tfrac56\}\); in the \(\ALG^*\) case one must have \(\ALG^*\text{-}I_k^*\) with \(k \in \{1,2,3,4\}\); and in the \(\ALH^*\) case one must have \(k \in \{1,\dots,9\}\).

Furthermore, the hyperk\"ahler metrics can be obtained via a generalized Hein construction by incorporating additional parameters into \cite{Hein}, except possibly in the $\ALH$ case (see the discussion in \cite{Chen-Chen3}. A common strategy in these compactification results is to construct a holomorphic function realizing \((M,J)\) as an elliptic fibration over \(\mathbb{C}\). Such a function \(f\) can be readily constructed on the asymptotic model for a suitable choice of \(J\). One then grafts \(f\) to \(M\) and uses weighted Fredholm theory to first produce a harmonic function \(\widetilde{f}\) asymptotic to \(f\), and subsequently prove that \(\widetilde{f}\) is in fact holomorphic.

In the \(\ALH^*\) case, it is proved in \cite{HSVZ2} (see also \cite{LinCollins2024SYZSurvey}) that, for a different choice of complex structure \(I\), the manifold \((M,I)\) admits a compactification to a weak del Pezzo surface, and the hyperk\"ahler metric arises from the construction in \cite{TianYau1990}. 

This compactification is more delicate, since the natural \(I\)-holomorphic functions on the asymptotic model exhibit super-exponential growth at infinity. As a result, the strategy based on weighted Fredholm theory encounters difficulties, due to the fact that in the construction of \cite{TianYau1990}, the decay rate of the metric is much slower than that of the complex structure. 
In \cite{HSVZ2}, a new approach is introduced, based on the more robust \(L^2\)-estimates in several complex variables pioneered by H\"ormander. This method also admits higher-dimensional generalizations, relating asymptotically Calabi metrics to weak Fano manifolds; see \cite{HSVZ2, chen2024calabiyaumetricscalabitype}.

\subsection{Torelli type theorems}
Torelli-type theorems for hyperk\"ahler gravitational instantons are motivated by the corresponding results for hyperk\"ahler metrics on the \emph{K3 manifold}. Recall that the K3 manifold \(\mathcal K\) is, by definition, the underlying oriented smooth 4-manifold of a complex K3 surface. It is known that \(\mathcal K\) is simply connected, and the cup product on \(H^2(\mathcal K; \mathbb{Z})\) has signature \((3,19)\).
Denote by
\[
\mathcal D = \OO(3,19)/(\OO(3)\times \OO(19))
\]
the Grassmannian of \emph{positive} 3-dimensional subspaces of \(H^2(\mathcal K; \mathbb{R})\). There is a natural action of the automorphism group \(\Aut(H^2(\mathcal K; \mathbb{Z}))\) on \(\mathcal D\).

 Denote by $\widehat{\mathcal M}$ the set of hyperk\"ahler metrics on $\mathcal K$ with unit diameter. Then we have a \emph{period map}
  \begin{equation}
 \mathcal P: \widehat{\mathcal M}\rightarrow \mathcal D; g\mapsto  \mathbb H^+_g, 
 \end{equation}
 where $\mathbb H^+_g\subset H^2(\mathcal K, \mathbb R)$ denotes the three dimensional subspace of $H^2(\mathcal K; \R)$ represented by the span of the hyperk\"ahler triple $\{\omega_1, \omega_2, \omega_3\}$. The Torelli theorem in this case states that 
 \begin{itemize}
 	\item(Injectivity) $g_1$ and $g_2$ are isometric, i.e., $g_2=\varphi^*g_1$ for  $\varphi\in \text{Diff}(\mathcal K)$,
 if and only if $\mathcal P(g_1)=\gamma.\mathcal P(g_2)$ for some $\gamma\in \Aut(H^2(\mathcal K; \mathbb Z))$.
  	\item(Surjectivity) $\mathcal P$ is surjective onto the open subset $\mathcal D^\circ\subset\mathcal D$ of \emph{non-degenerate} periods.  Here a 3-dimensional positive subspace $V$ of $H^2(\mathcal K; \R)$ is said to be non-degenerate if for any  $\delta\in H_2(\mathcal K; \mathbb Z)$ with $\delta.\delta=-2$, there is an element $v\in V$ such that $\int_{\delta}v\neq0$.
 \end{itemize}

Torelli-type results for hyperk\"ahler gravitational instantons have been established by Kronheimer \cite{Kronheimer89T} in the \(\ALE\) case, by Chen-Chen \cite{Chen-Chen2, Chen-Chen3} in the \(\ALF\) case, by Chen-Viaclovsky-Zhang \cite{CVZ} and Lee-Lin \cite{LL} in the \(\ALG\) and \(\ALG^*\) cases, and by Collins-Jacob-Lin \cite{CJL1} and Lee-Lin \cite{LL} in the \(\ALH^*\) case.

In the \(\ALE\) case, the underlying smooth manifold \(\mathcal E\) is the minimal resolution of \(\mathbb{C}^2/\Gamma\) for some \(\Gamma\). The period domain \(\mathcal D\) is given by the Grassmannian of all 3-dimensional subspaces of \(H^2(\mathcal E; \mathbb{R})\), together with a suitable non-degeneracy condition. In the \(\ALF\) case, the situation is similar, but there is an additional continuous parameter corresponding to the asymptotic length of the circle fibration.

In the remaining cases, the statements are more intricate, and we refer to the references above for precise formulations. The proofs of the Torelli theorems in these settings rely crucially on the compactification to rational elliptic surfaces discussed above, as well as on gluing constructions involving K3 manifolds, which we describe below.

\subsection{Connections to compactfied moduli space of K3}\label{ss:GI and K3}
Denote by $$\mathcal M:=\widehat{\mathcal M}/\text{Diff}(\mathcal K)$$ the  space of isometry classes of unit diameter hyperk\"ahler metrics on $\mathcal K$. There is a natural Gromov-Hausdorff compactification $\overline{\mathcal M}^{\GH}$. There is a subset $\overline{\mathcal M}^{\partial}$ consisting of non-collapsed Gromov-Hausdorff limits, and the complement $\overline{\mathcal M}^{\GH}\setminus \overline{\mathcal M}^{\partial}$ consists of \emph{collapsed} Gromov-Hausdorff limits. 
Hyperk\"ahler gravitational instantons are closely related to understanding $\overline{\mathcal M}^{\GH}$ and the singularity formation of hyperk\"ahler metrics in $\mathcal M$. 

The \emph{gluing construction} allows one to realize certain hyperk\"ahler gravitational instantons as \emph{bubble limits}, i.e., rescaled pointed Gromov-Hausdorff limits, of degenerating hyperk\"ahler metrics on \(\mathcal K\). A prototypical example is provided by the \emph{metric} Kummer construction.

One begins with a flat 4-torus \(\mathbb{T}^4\) and considers its \(\mathbb{Z}_2\)-quotient by the standard involution. The quotient \(\mathbb{T}^4/\mathbb{Z}_2\) is then a hyperk\"ahler orbifold with 16 singularities, each locally modeled on the flat cone \(\mathbb{R}^4/\mathbb{Z}_2\). One can resolve these singularities by gluing in 16 copies of the Calabi-Eguchi-Hanson space, suitably rescaled so that the totally geodesic \((-2)\)-spheres have small area. A quantitative implicit function theorem is then applied to deform the resulting metric to a hyperk\"ahler metric on \(\mathcal K\).

As the areas of these \((-2)\)-spheres tend to zero, the metrics converge in the Gromov-Hausdorff sense to the orbifold limit. In particular, \(\mathbb{T}^4/\mathbb{Z}_2\) lies in \(\overline{\mathcal M}^{\partial}\), and the Calabi-Eguchi-Hanson space appears as a bubble limit of hyperk\"ahler metrics on \(\mathcal K\). 

More generally, gluing constructions for $\ALE$ hyperk\"ahler gravitational instantons are well studied (see, for example, \cite{Topiwala1987KEK3, Donaldson2012Kummer, OdakaOshima2021CollapsingK3}). The upshot is that, given any hyperk\"ahler orbifold, there exists a gluing construction that realizes it as an element of $\overline{\mathcal M}^\partial$, and $\ALE$ gravitational instantons arise as bubble limits. See also \cite{tazoe2025bubblinglimitsnoncollapsing} for determining these bubbles limits in case of algebraic degenerations of polarized K3 surfaces, in terms of period mappings. 

However, not all $\ALE$ gravitational instantons can arise in this way, since $\mathcal K$ has fixed topology. For example, if an $\ALE$-$A_k$ gravitational instanton bubbles off, then $k \leq 16$. It is therefore an interesting problem to determine which $\ALE$ gravitational instantons can occur; equivalently, this amounts to determining which $ADE$ singularities can appear on an orbifold $K3$ surface. 

\

There are extensive results on constructing examples of collapsing hyperk\"ahler metrics on $\mathcal K$. Since $\mathbb T^4/\mathbb Z_2$ lies in $\overline{\mathcal M}^{\GH}$ and clearly admits collapse to lower-dimensional spaces of the form $\mathbb T^k/\mathbb Z_2$ for $k = 1,2,3$, it follows that these limits also belong to $\overline{\mathcal M}^{\GH}$. More systematic constructions require gluing non-$\ALE$ hyperk\"ahler gravitational instantons.

More general three-dimensional limits in $\overline{\mathcal M}^{\GH}$ were constructed by Foscolo \cite{Fos19}, where the nontrivial bubble limits are given by $\ALF$-$A_k$ or $D_k$ gravitational instantons. Again, there are bounds on $k$ arising from the fixed topology.

Two-dimensional limits in $\overline{\mathcal M}^{\GH}$ can be obtained by collapsing the fibers of an elliptic $K3$ surface, beginning with the work of Gross-Wilson \cite{GW00}, which considers an elliptic $K3$ surface with $24$ singular fibers of Kodaira type $I_1$. In this setting, away from the singular fibers the metrics are modeled on \emph{semi-flat metrics} introduced by Greene-Shapere-Vafa-Yau \cite{GreeneShapereVafaYau1990Stringy}, while the nontrivial bubble limits are given by the Taub-NUT metrics.

An interesting feature is the presence of an intermediate scale at which one observes an incomplete model metric, namely the \emph{Ooguri-Vafa metric}. This metric can be constructed by applying the Gibbons-Hawking ansatz to the Green's function on a domain in $\mathbb S^1 \times \mathbb R^2$ (note that this space is non-parabolic, so the Green's function cannot be globally positive).

The case of collapsing general elliptic $K3$ surfaces was studied in \cite{CVZ20} (see also earlier work \cite{GrossTosattiZhang2016GH, OdakaOshima2021CollapsingK3}), where the bubble limits around singular fibers with finite monodromy are shown to be $\ALG$ gravitational instantons.

A one-dimensional limit $(X_\infty, d_\infty)$ in $\overline{\mathcal M}^{\GH}$ must be isometric to the unit interval $[0,1]$. However, general theory also endows $X_\infty$ with additional structure, namely a \emph{renormalized limit measure} $\nu_\infty$. As shown in \cite{HondaSunZhang2019CollapsingHK}, there is a canonical affine structure on $X_\infty$, and with respect to an affine coordinate $z$ one can write the metric as
$$
g_\infty = c_1 L \, dz^2,
$$
and the renormalized limit measure as
$$
\nu_\infty = c_2 L \, dz,
$$
where $L$ is a piecewise linear function of $z$ and $c_1, c_2 > 0$ are constants.

The case where $L$ is constant was constructed by Chen-Chen \cite{Chen-Chen3} by gluing two $\ALH$ hyperk\"ahler gravitational instantons. From the complex-geometric point of view, this corresponds to doubling a rational elliptic surface to obtain an elliptic $K3$ surface.

The case of non-constant $L$ was constructed by Hein-Sun-Viaclovsky-Zhang \cite{HSVZ22} via a gluing construction of two $\ALH^*$ gravitational instantons with a nontrivial neck region. This neck region can be viewed as a counterpart of the Ooguri-Vafa metric, and can be constructed by applying the Gibbons-Hawking ansatz to the Green's function on a domain in $\mathbb T^2 \times \mathbb R$.

From the complex-geometric point of view, it is shown in \cite{SZ19} that this models Type $\mathrm{II}$ degenerations of polarized $K3$ surfaces. In fact, a similar picture extends to the so-called \emph{small complex structure limits} of higher-dimensional Calabi-Yau manifolds, although the construction of the neck region becomes significantly more involved.
A similar neck region was subsequently used in \cite{CVZ} to show that $\ALG^*$ gravitational instantons can arise as bubble limits of hyperk\"ahler metrics on $\mathcal K$.

As a converse to the above gluing constructions, the collapsing theory developed in \cite{SZ} provides a description of families of hyperk\"ahler metrics in $\mathcal M$ that undergo collapse. First, the possible limits in $\overline{\mathcal M}^{\GH} \setminus \overline{\mathcal M}^\partial$ are classified: they are either a flat $T^3/\mathbb Z_2$, a singular special K\"ahler metric on $S^2$ (with integral monodromy, as observed in \cite{Ouyang2025CollapsingK3Special}), or the unit interval $[0,1]$.

Furthermore, after passing to a subsequence, the collapse over the regular set of the limit occurs along a nilpotent fiber bundle, and the metrics can be perturbed to hyperk\"ahler metrics that are invariant under the nilpotent group action. The analysis is similar to that carried out in the proof of Theorem~\ref{thm:SZ}.

The collapsing theory developed in \cite{SZ} also provides a framework for controlling the period map $\mathcal P$ as one approaches the boundary of $\overline{\mathcal M}^{\GH}$. In particular, a strategy (see \cite{SZ}, Section 7) was proposed to study the relationship between $\overline{\mathcal M}^{\GH}$ and certain Satake compactifications of $\Aut(H^2(\mathcal K; \mathbb Z)) \setminus \mathcal D$. It is also possible to show that any  
degenerating sequence in $\mathcal M$ is given by a suitable generalization of the known gluing constructions, based on the analysis of bubble limits and classification of gravitational instantons (with possible orbifold singularities). 

Foundational work in this direction was carried out by Odaka-Oshima \cite{OdakaOshima2021CollapsingK3, odaka2020polystablelogcalabiyauvarieties, odaka2020pldensityinvarianttype}, where precise conjectural relationships were formulated and partially established. The techniques and strategy of \cite{SZ} were subsequently applied in \cite{Liu} to give a simple proof of the surjectivity of the period map for $K3$ surfaces, and later in \cite{OuyangTian2025CompactificationK3} to complete the proof of a conjecture from \cite{OdakaOshima2021CollapsingK3}.
It is plausible that these can also be applied to the study of moduli spaces of gravitational instantons (see Problem~\ref{p:moduli of GI}). For example, it is tempting to give a new proof of the Torelli theorem for hyperk\"ahler gravitational instantons, as in \cite{Liu}.

\subsection{General Type \texorpdfstring{$\I$}{I} case}\label{ss:general Type I}
As mentioned earlier, a Type $\mathrm{I}$ gravitational instanton $(M,g)$ need not be globally hyperk\"ahler.

If $\pi_1(M)$ is finite, which is the case when $(M,g)$ is non-collapsed, then the universal cover $(\widetilde M, \widetilde g)$ is a hyperk\"ahler gravitational instanton, and one may appeal to the classification theorem discussed above. The problem then reduces to studying free actions of finite groups on hyperk\"ahler gravitational instantons. This is completely understood in the non-collapsed case by the results of Souvaina \cite{Souvaina} and Wright \cite{Wright}. In a suitable complex structure, such spaces arise as smoothings of $\mathbb C^2/\Gamma$ for $\Gamma \subset U(2)$.

When $\pi_1(M)$ is infinite, the situation is more subtle.  Its universal cover carries a complete hyperk\"ahler metric, but can not have quadratic curvature decay, unless $(M, g)$ is flat.  At present, all known such examples are flat, so given by a quotient of $\R^4$.  See Question~\ref{q:Type I infinite fund}.

\section{Type \texorpdfstring{$\II$}{II} gravitational instantons} \label{s:Type II}

In this section we discuss Type II gravitational instantons. Such a metric $(M, g)$ is \textit{locally} Hermitian but \emph{not} locally K\"ahler. In the following we will focus on the case when $(M, g)$ is  Hermitian. As explained before, in general one may need to pass to a double cover. 

As in Section \ref{ss:Type II case} we denote by $g_K=\lambda^{2/3}g$ the associated K\"ahler metric.  A straightforward computation yields the equation
\begin{equation}\label{eqn:scalar curvature equation}6\Sc_{g_K}\Delta_{g_K}\Sc_{g_K}-12|\nabla_{g_K}\Sc_{g_K}|_{g_K}^2+\Sc_{g_K}^3=0.
\end{equation}
Since the curvature of $g$ decays at infinity, we know that the function $$|\Sc_{g_K}|=\lambda^{1/3}=(2\sqrt{6}|\W^{+}|_g)^{1/3}$$ tends to zero at infinity. A simple maximum principle argument then implies that $\Sc_{g_{K}}$ must be everywhere positive, and thus $$\Sc_{g_K}=\lambda^{1/3}.$$

\subsection{LeBrun-Tod ansatz}
\label{ss:lebrun-tod}
Given a Hermitian gravitational instanton $(M,g)$, one can locally perform a K\"ahler reduction to obtain an associated conformal K\"ahler metric $g_K$ with respect to the vector field $\mathbb K$. The corresponding moment map is given by $\xi = \Sc_{g_K}$. Let $x + i y$ be a local holomorphic coordinate on the complex quotient. Then, away from the zero set of $\mathbb K$, the metric $g_K$ can be locally expressed via the LeBrun-Tod ansatz \cite{LeBrun91, Tod26}.

Denote by $\eta$ the dual 1-form to $\mathbb K$, and set $$W:=|\mathbb K|_{g_K}^{-2}.$$ Then $g_K$ can be written in the form
\begin{equation}\label{eq:lebrun-tod kahler metric}
    g_K=Wd\xi^2+W^{-1}\eta^2+We^v(dx^2+dy^2),
\end{equation}
with 
\begin{equation}\label{eqn:linearized Toda}
    (We^v)_{\xi\xi}+W_{xx}+W_{yy}=0.
\end{equation}
Here,  $\eta$ satisfies that
\begin{equation}
    d\eta=(We^v)_{\xi} dxdy+W_xdyd\xi+W_yd\xi dx,
\end{equation}
and the complex structure is given by 
$$J:d\xi\mapsto -W^{-1}\eta,\ dx\mapsto -dy.$$
One can check (see \cite{BG}) that equation \eqref{eqn:linearized Toda} together with the condition that $g_K$ is conformal to a Ricci-flat metric are equivalent to that $v$  satisfies a \emph{twisted $SU(\infty)$ Toda }equation
\begin{equation}\label{eq:twisted toda equation}
    (e^v)_{\xi\xi}+v_{xx}+v_{yy}=-\xi We^v,
\end{equation}
and $W$ is determined by
\begin{equation}\label{eq:W in lebrun-tod}
    W=\frac{12}{\xi^3}-\frac{6v_\xi}{\xi^2}.
\end{equation}
By setting 
$$\varrho:=1/\xi,\ V:=\xi^2W,\ u:=v-4\log\xi,$$
one can also rewrite the above ansatz as
\begin{equation}\label{eq:lebrun-tod ricci flat metric}
    g=\varrho^2 g_K=V(d\varrho^2+e^u(dx^2+dy^2))+V^{-1}\eta^2,
\end{equation}
with $u$ satisfying the usual $SU(\infty)$ Toda equation
\begin{equation}\label{eq:toda equation}
    (e^u)_{\varrho\varrho}+u_{xx}+u_{yy}=0,
\end{equation}
and $V$ given by
\begin{equation}\label{eq:V in lebrun-tod}
    V=-12\varrho+6\varrho^2u_{\varrho}.
\end{equation}

\subsection{Asymptotic models}

In this subsection, we introduce some Ricci-flat model ends which are Hermitian.

\subsubsection{$\ALE$ models}

For a finite subgroup of $\Gamma\subset U(2)$ acting freely on $S^3$, we have the standard flat K\"ahler metric $g_\Gamma$ on $(\mathbb{C}^2\setminus\{0\})/\Gamma$. Now the conformal metric $$\bar g_\Gamma:=\frac{1}{r^4}g_{\mathbb{C}^2/\Gamma}$$
is also flat; indeed, under the isometry explicitly given by
$$y\mapsto \frac{y}{|y|^2},$$
$(\mathbb{C}^2\setminus\{0\})/\Gamma$ with the conformal metric $\bar g_{\Gamma}$ can be isometrically identified with $(\mathbb{R}^4\setminus\{0\})/\Gamma$, but the complex structure becomes  a non-standard one. In this way an $\ALE$-$\Gamma$ model end is Hermitian.

\subsubsection{$\ALF$ models}

Recall that the $\ALF$-$A_k$ $(k\in \mathbb Z)$ hyperk\"ahler model end is given by the Gibbons-Hawking ansatz applied to $\R^3$ with the harmonic function 
$$V=1+\frac{k+1}{2r}.$$  It has an explicit expression as
$$g=(1+\frac{k+1}{2r})(dr^2+r^2(d\theta^2+\sin^2\theta d\phi^2))+(1+\frac{k+1}{2r})^{-1}(dt+\frac{(k+1)\cos\theta}{2}d\phi)^2, \quad  r\in (0, \infty).$$
Here $0\leq\theta\leq\pi,0\leq\phi\leq2\pi$, and $0\leq t\leq2\pi$.

If we reverse the orientation, then the metric $g$ becomes Hermitian. Indeed, one can write down the associated complex structure explicitly as
$$J:dr\mapsto(1+\frac{k+1}{2r})^{-1}(dt+\frac{(k+1)\cos\theta}{2}d\phi),\ d\theta\mapsto-\sin\theta d\phi,$$
and the corresponding conformal extremal K\"ahler metric is given by 
$$g_K=(r+\frac{k+1}{2})^{-2}g.$$ 
In this way an $\ALF$-$A_{k}$ model end, with the reversed orientation, is Hermitian. A finite isometric quotient of this, which preserves the complex structure $J$, is then an $\ALF$ model end for Hermitian gravitational instantons.

Notice that the hyperk\"ahler $\ALF$-$D_k$ model ends, with the reversed orientation, are not Hermitian,  since the $\mathbb Z_2$ action does not preserve the complex structure $J$. On the other hand, there are $\ALF$ model ends which do not fall into the $A_k$ family.  As an example, one can consider the isometric quotient of the $\ALF$-$A_k$ end by the action $\phi\mapsto \phi+\pi,t\mapsto t+\pi$.

As a side remark, for the hyperk\"ahler $\ALF$-$A_k$ model end, there is another complex structure given by 
$$J':dr\mapsto-(1+\frac{k+1}{2r})^{-1}(dt+\frac{(k+1)\cos\theta}{2}d\phi),\ d\theta\mapsto-\sin\theta d\phi,$$
such that  the conformal metric $r^{-2}g$ is  scalar-flat K\"ahler.

\begin{remark}
 In \cite{LiSun2025}, the $\ALF$ model ends are divided into $\ALF^+$ and $\ALF^{-}$ types, corresponding to the hyperk\"ahler and Hermitian orientations, respectively. In this survey, we do not make this distinction.
\end{remark}

\subsubsection{$\AF$ models}
For $\beta\in [0, 1)$, an $\AF_\beta$ model end is the quotient of the form $$\mathbb R^4/\mathbb Z=(\mathbb R^3\times \mathbb R^1)/\mathbb Z=\mathbb{R}^1\times(\mathbb{R}^2\times\mathbb{R}^1)/\mathbb{Z}, $$ where the generator of $\mathbb{Z}$ acts as a rotation by $2\pi\beta$ on $\mathbb{R}^2$ and a translation by $1$ on $\mathbb{R}^1$.  When $\beta=0$ this agrees with the $\ALF$-$A_{-1}$ model. For $\beta\neq 0$, the $\AFa$ model is locally hyperk\"ahler, but not hyperk\"ahler. 

The $\AFa$ model ends exhibit interesting asymptotic behavior at infinity. When $\beta = p/q$ for coprime integers $p,q$, the space has an asymptotic cone given by $\mathbb C_{1/q} \times \mathbb R$, where $\mathbb C_{1/q}$ denotes the two-dimensional flat cone with cone angle $2\pi/q$. When $\beta \notin \mathbb Q$, the space instead has an asymptotic cone given by the two-dimensional flat half-plane $\mathbb H$.

With respect to the decomposition $\mathbb R^4 = \mathbb R^3 \times \mathbb R$, let $\rho$ denote the radial coordinate on $\mathbb R^3$ and $t$ the coordinate on the $\mathbb R$ factor. Let $\mathbb S^2$ be the unit sphere in $\mathbb R^3$. Then the flat metric on $\mathbb R^4$ can be written as
$$
d\rho^2 + \rho^2 g_{\mathbb S^2} + dt^2.
$$
There exists a complex structure $J$ such that $J(d\rho) = -dt$, and whose restriction to $\mathbb S^2$ agrees with the standard complex structure. With respect to $J$, the conformally rescaled metric $\rho^{-2} g_{\mathbb R^4}$ is K\"ahler. In particular, the $\AF_\beta$ model end is Hermitian.
\ 

There are two non-standard asymptotic models that can be obtained from explicit solutions of \eqref{eq:toda equation}.

\subsubsection{Special Kasner models}
Let $u = \log(-\varrho)$, and define on $(-\infty,-1) \times \mathbb{R}^3$ the \emph{special Kasner metric} by
\begin{equation}\label{kasner}
g_{\text{Kasner}} = -6\varrho \bigl(d\varrho^2 - \varrho (dx^2 + dy^2)\bigr) - \frac{1}{6\varrho} \, dt^2.
\end{equation}
A special Kasner model end is obtained by taking the quotient of this space by a lattice $\Lambda \subset \mathbb{R}^3$. This metric is Hermitian with respect to both orientations.

This Type $\mathrm{II}$ metric belongs to the Kasner family of Ricci-flat metrics, which are of the form
\begin{equation}\label{eqn:Kasner}
dr^2 + r^{2p_1} d\theta_1^2 + r^{2p_2} d\theta_2^2 + r^{2p_3} d\theta_3^2,
\end{equation}
where $p_1, p_2, p_3 \in \mathbb{R}$ satisfy
\begin{equation}
    p_1 + p_2 + p_3 = 1, \qquad p_1^2 + p_2^2 + p_3^2 = 1.
\end{equation}
The metric in \eqref{kasner} corresponds to the case $p_1 = p_2 = \tfrac{2}{3}$ and $p_3 = -\tfrac{1}{3}$.

It has quadratic volume growth and exactly quadratic curvature decay.	

\subsubsection{Anti-$\ALH^*$ models}
We take $u = 0$ and define on $(-\infty,-1) \times \mathbb{R}^3$ the metric
\begin{equation}\label{Anti ALH* model}
g_{\ALH^*} = -12\varrho \bigl(d\varrho^2 + dx^2 + dy^2\bigr) - \frac{1}{12\varrho}\bigl(dt - 12x\,dy\bigr)^2.
\end{equation}
This metric is isometric to the hyperk\"ahler $\ALH^*$ model defined earlier, but is Hermitian with respect to the opposite orientation. An anti-$\ALH^*$ model is obtained by taking a quotient of $g_{\ALH^*}$ by a cocompact lattice $\Gamma$ in the Heisenberg group.

It has $\frac{4}{3}$ volume growth rate and exactly quadratic curvature decay.

\begin{remark}
    Notice that for the special Kasner model and the anti-$\ALH^*$ models, the scalar curvature of the associated conformal K\"ahler metric $g_K$ is negative. So by the discussion at the beginning of this section, they cannot serve as the asymptotic model of a Hermitian gravitational instanton. 
\end{remark}

\subsection{Examples}

Some hyperk\"ahler gravitational instantons can be viewed as Type $\mathrm{II}$ metrics with respect to the reversed orientation. This includes the Calabi-Eguchi-Hanson metric on $T^*S^2$ with reversed orientation, which is $\ALE$, and the anti-Taub-NUT metric on $\mathbb R^4$, which is $\ALF$.

\

 Page \cite{PAGE1978249} discovered a “Taub-NUT instanton with horizon,” now commonly referred to as the \emph{Taub-Bolt metric}. In this construction, the \emph{nut} (i.e., the fixed point of the $S^1$ action on $\mathbb R^4$) is replaced by a \emph{bolt}, namely an $S^2$. The underlying topology is $\mathbb{CP}^2 \setminus \{\mathrm{pt}\}$. The metric is of cohomogeneity one and admits the explicit form
\begin{equation}
ds^{2}
= \frac{r^{2}-1}{r^{2}+1-\frac{5}{2}r}\,dr^{2}
+ \frac{r^{2}+1-\frac{5}{2}r}{r^{2}-1}\,(d\psi + 2\cos\theta\, d\phi)^{2}
+ (r^{2}-1)\,(d\theta^{2}+\sin^{2}\theta\, d\phi^{2}).
\tag{4.35}
\end{equation}
The coordinates $(r,\theta)$ satisfy $r \ge r_{0}$ and $0 \le \theta \le \pi$, where $r_{0}$ is the larger root of $r^{2}+1-\frac{5}{2}r$. The coordinates $(\psi,\phi)$ satisfy the periodicity conditions
$$
(\psi,\phi)\sim(\psi+8\pi,\phi), \qquad (\psi,\phi)\sim(\psi+4\pi,\phi+2\pi).
$$
One can verify that the Taub-Bolt metric is of Type $\mathrm{II}$ for both orientations. It is asymptotic to the Taub-NUT metric and is therefore $\ALF$. With respect to the orientations $dr \wedge d\theta \wedge d\phi \wedge d\psi$ and $-dr \wedge d\theta \wedge d\phi \wedge d\psi$, the underlying complex surfaces can be identified with $\mathcal{O}(-1)$ and $\mathcal{O}(1)$, respectively. We refer to the former as the anti-Taub-Bolt gravitational instanton and the latter as the Taub-Bolt gravitational instanton.

\

The \emph{Kerr} family of gravitational instantons arises from the Kerr solutions to the vacuum Einstein equations via a Wick rotation (i.e., replacing $t \mapsto i t$) \cite{GibbonsHawking1977}. Modulo scaling, this yields a one-parameter family of metrics on $S^2 \times \mathbb R^2$, parametrized by $\beta \in [0,1)$, each of which is asymptotic to the $\AF_\beta$ model end. When $\beta = 0$, one recovers the \emph{Schwarzschild} metric.

An explicit formula is available (see, for example, \cite{ChenTeo2}):
\begin{equation}
g_a = \frac{\Delta}{\Sigma} (d\phi_1 + a \sin^2 \vartheta \, d\phi_2)^2 
+ \frac{\sin^2 \vartheta}{\Sigma} \bigl(a \, d\phi_1 - (s^2 - a^2) d\phi_2\bigr)^2  
+ \Sigma\left(\frac{1}{\Delta} ds^2 + d\vartheta^2\right),
\label{eq:kerr_instanton}
\end{equation}
where
\begin{equation}
\Sigma = s^2 - a^2 \cos^2 \vartheta, \qquad \Delta = s^2 - 2ms - a^2.
\end{equation}
Here $a \in \mathbb R$, $s \ge s_a$, and $\vartheta \in [0,\pi]$, where
\[
s_a = m + \sqrt{m^2 + a^2}.
\]
We fix $m = \tfrac{1}{2}$ to remove the scaling freedom.

One can verify that the Kerr family consists of Type $\mathrm{II}$ metrics for both orientations.

\

A surprising discovery by Chen-Teo \cite{Chen-Teo} is the existence of a one-parameter family (modulo scaling) of gravitational instantons on $(S^2 \times \mathbb R^2)\sharp \overline{\mathbb{CP}}^2$, parametrized by $\beta \in (0,1)$. These metrics are constructed via the inverse scattering transform and admit an explicit expression in terms of rational functions. It was reconstructed using toric K\"ahler geometry by Biquard-Gauduchon \cite{BG}, see the discussion in Section \ref{ss:BG}. There is also the study of Chen-Teo gravitational instanton from the twistor theory point of view  \cite{DunajskiTod2024TwistorChenTeo}.

It was shown by Aksteiner-Andersson \cite{AA} that the Chen-Teo metrics are of Type $\mathrm{II}$ for exactly one choice of orientation. With this choice, they are referred to as the Chen-Teo gravitational instantons.

\subsection{The results of Biquard-Gauduchon}\label{ss:BG}

One observes that all the examples above are \emph{toric}, in the sense that there is a $\mathbb T^2$-action preserving both the metric and the complex structure. In \cite{BG}, Biquard-Gauduchon provided a unified construction and interpretation of toric Type $\mathrm{II}$ gravitational instantons with $\ALF$ or $\AF$ asymptotics.  In this subsection, we briefly sketch their construction. Note that the terminology in \cite{BG} differs from ours: both $\ALF$ and $\AF$ asymptotics are referred to there as $\ALF$.

As we have seen in Section~\ref{ss:lebrun-tod}, a Hermitian gravitational instanton $(M,g)$ already admits a Killing field $\mathbb{K}$. Under an additional symmetry, by choosing the coordinate $x+iy$ appropriately, one can further assume in the LeBrun-Tod ansatz \eqref{eq:lebrun-tod ricci flat metric}-\eqref{eq:V in lebrun-tod} that the functions $u$ and $V$ are independent of the variable $y$. As observed by Ward \cite{Ward}, a B\"acklund transformation then shows that local solutions to the reduced $SU(\infty)$ Toda equation
\begin{equation}\label{eqn:reduced Toda}
(e^u)_{\varrho\varrho} + u_{xx} = 0
\end{equation}
correspond to axisymmetric harmonic functions $U$ on $\mathbb{R}^3$, i.e.,
\begin{equation}\label{eqn:axisymmetric harmonic}
\frac{1}{\rho}(\rho U_\rho)_{\rho} + U_{zz} = 0,
\end{equation}
where $\rho$ denotes the distance to the $z$-axis. The correspondence is given by
$$
\varrho = \tfrac{1}{2}\rho U_\rho, \qquad x = -\tfrac{1}{2}U_z, \qquad u = \log \rho^2.
$$
Conversely, any local solution $u$ to \eqref{eqn:reduced Toda} arises from such a construction. Note that \eqref{eqn:reduced Toda} is nonlinear, whereas \eqref{eqn:axisymmetric harmonic} is linear.

Therefore, in the toric setting, the LeBrun-Tod ansatz can be further reduced to axisymmetric harmonic functions. Recall that the ansatz \eqref{eq:lebrun-tod ricci flat metric}-\eqref{eq:V in lebrun-tod} is obtained by performing K\"ahler reduction with respect to the Killing field $\mathbb{K}$. To align with the notation of Biquard-Gauduchon \cite{BG}, we instead consider the rescaled Killing field
$$
X := \sqrt{6k}\,\mathbb{K},
$$
for a nonzero constant $k$. The only change in the formulas \eqref{eq:lebrun-tod ricci flat metric}-\eqref{eq:V in lebrun-tod} is that \eqref{eq:V in lebrun-tod} becomes
\begin{equation}
    V = \frac{1}{k}\bigl(-2\varrho + 6\varrho^2 u_{\varrho}\bigr).
\end{equation}

By the above discussion, starting from an axisymmetric harmonic function $U$, one can write the Ricci-flat metric $g$ in the form
\begin{equation}\label{eq:harmonic functioin ansatz}
    g = e^{2\nu}(d\rho^2 + dz^2) + V\rho^2 dy^2 + V^{-1}(dt - F\,dy)^2,
\end{equation}
where
\begin{align}
    e^{2\nu} &= \frac{1}{4} V\rho^2\bigl(U_{\rho z}^2 + U_{zz}^2\bigr),\\
    V &= -\frac{1}{k}\left(\rho U_\rho + \frac{U_\rho^2 U_{zz}}{U_{\rho z}^2 + U_{zz}^2}\right),\\
    F &= -\frac{1}{k}\left(-\frac{\rho U_\rho^2 U_{\rho z}}{U_{\rho z}^2 + U_{zz}^2} + \rho^2 U_z + 2H\right).
\end{align}
Here $H$ is a harmonic conjugate of $U$, and these expressions are valid on the domain where $V > 0$ and $U_{\rho z}^2 + U_{zz}^2 > 0$.

To understand the global geometry, one must impose appropriate behavior of the axisymmetric harmonic function $U$ near the $z$-axis in $\mathbb{R}^3$. This was analyzed by Biquard-Gauduchon \cite{BG} using toric K\"ahler geometry. More precisely, the associated K\"ahler metric $(M,g_K)$ is toric, and hence corresponds to a convex unbounded Delzant polytope $\mathbf{\Delta} \subset \mathbb{R}^2$, together with a K\"ahler potential defined on its interior $\mathbf{\Delta}^\circ$. The smoothness of the metric and the $\ALF/\AF$ asymptotics are encoded in the behavior of the potential near $\partial \mathbf{\Delta}$ and at infinity.

It was shown that $(\rho,z)$ defines a diffeomorphism from $\mathbf{\Delta}$ onto the right half-plane $\mathbb{H}$, which can be viewed as the quotient of $\mathbb{R}^3$ by rotations about the $z$-axis. Thus, the axisymmetric harmonic function $U$ is globally defined on $\mathbb{H}^\circ$. By relating the above ansatz to the K\"ahler potential, it was proved in \cite{BG} that the behavior of $U$ near the $z$-axis is given as follows:
\begin{proposition}
There exists a convex piecewise linear function $f(z) > 0$, whose slope increases from $-1$ to $1$, such that near the $z$-axis, $U - f(z)\log \rho^2$ remains bounded.
\end{proposition}

Biquard-Gauduchon \cite{BG} then introduced the basic axisymmetric harmonic function
$$
U_0(\rho,z) = 2r - z \log \frac{r+z}{r-z},
$$
which corresponds geometrically to the flat metric and has asymptotic behavior near the $z$-axis determined by the piecewise linear function
$$
f(z) = |z|.
$$
The behavior of $U$ at infinity can similarly be shown to satisfy
$$
U - U_0 = o(r).
$$
By superposing such basic solutions, it follows that for any convex piecewise linear function
$$
f(z) = A + \sum_{j=1}^n a_j |z - z_j|,
$$
with $A > 0$, $z_1 < \cdots < z_n$, and slope increasing from $-1$ to $1$ (so that each $a_j \in (0,1)$ and $\sum a_j = 1$), one can associate a canonical axisymmetric harmonic function
\begin{equation}\label{eq:axisymmetric harmonic U}
    U = A \log \rho^2 + \sum_{j=1}^n a_j U_0(\rho, z - z_j).
\end{equation}
Using the ansatz \eqref{eq:harmonic functioin ansatz} and choosing $k = 2A$, which ensures that $|X|_g^2 = V^{-1} \to 1$ at infinity, one obtains a Hermitian Ricci-flat metric defined on $\mathbb{H}^\circ \times \mathbb{R}^2$, admitting an $\mathbb{R}^2$-translation symmetry. Taking the quotient of $\mathbb{R}^2$ by a suitable lattice $\Lambda$, the metric completion of $(\mathbb{H}^\circ \times (\mathbb{R}^2/\Lambda), g)$ yields a smooth $4$-manifold.

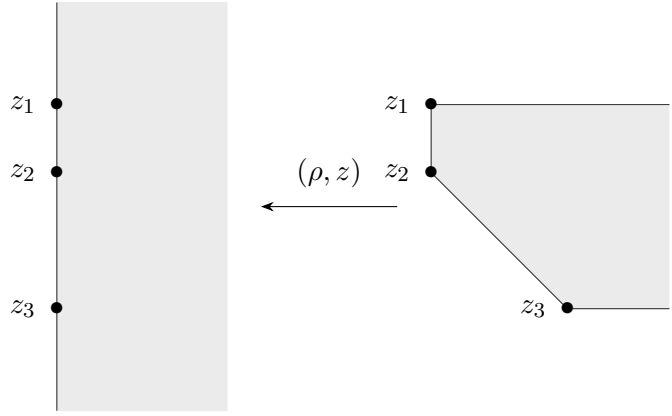
\begin{figure}[ht]
    \begin{tikzpicture}[scale=0.45]
        \draw (0,12)--(0,0);
        \filldraw[color=gray!50,fill=gray!50,opacity=0.3](0,12)--(0,0)--(5,0)--(5,12)--(0,12);
        \node at (0,9) {$\bullet$};
        \node at (-1,9) {$z_1$};
        \node at (0,7) {$\bullet$};
        \node at (-1,7) {$z_2$};
        \node at (0,3) {$\bullet$};
        \node at (-1,3) {$z_3$};

        \draw[-Stealth] (10,6)--(6,6);
        \node at (8,7) {$(\rho,z)$};

        \draw (18,9)--(11,9)--(11,7)--(15,3)--(18,3);
        \filldraw[color=gray!50,fill=gray!50,opacity=0.3](18,9)--(11,9)--(11,7)--(15,3)--(18,3);
        \node at (11,9) {$\bullet$};
        \node at (11,7) {$\bullet$};
        \node at (15,3) {$\bullet$};
        \node at (10,9) {$z_1$};
        \node at (10,7) {$z_2$};
        \node at (14,3) {$z_3$};

    \end{tikzpicture}
	\caption{The diffeomorphism $(\rho,z)$.}
	 \label{Figure:identification between H and polytope}   
\end{figure}

Each interval $(z_j, z_{j+1})$ along the $z$-axis is identified, via the diffeomorphism $(\rho,z)$, with an edge of the polytope $\mathbf{\Delta}$. Let $f_j'$ denote the slope of $f$ on $(z_j, z_{j+1})$. By standard toric geometry, each edge of $\mathbf{\Delta}$ determines a Killing field, which in our setting is a linear combination of $\partial_y$ and $\partial_t$ and vanishes along that edge. A direct computation shows that these vectors are explicitly given by
\begin{equation}\label{eq:BG normailzed rod vector}
    \overline{\mv}_j =
    \begin{cases}
         \overline{\mathfrak{v}}_j = f_j'(\partial_{y} + F_j \partial_t), & \text{if } f_j' \neq 0, \\[0.5em]
         \overline{\mathfrak{v}}_j = \dfrac{1}{A} f_j^2 \partial_t, & \text{if } f_j' = 0.
    \end{cases}
\end{equation}
Here the constants $F_j$ are determined by the following conditions:
\begin{itemize}
\item
\[
F_0 = -\frac{1}{A}\Bigl(A + \sum_{j=1}^n a_j z_j\Bigr)^2 + \frac{1}{A} \sum_{j=1}^n a_j z_j^2;
\]
\item if $f_j' \neq 0$, then
\begin{equation}
    F_j - F_{j-1} = \frac{1}{A} f_j^2\left(\frac{1}{f_j'} - \frac{1}{f_{j-1}'}\right);
\end{equation}
\item if $f_j' = 0$, then
\begin{equation}
    F_{j+1} - F_{j-1}
    = \frac{1}{A}\left(f_j^2\left(\frac{1}{f_{j+1}'} - \frac{1}{f_{j-1}'}\right) - 2(z_{j+1} - z_j) f_j\right).
\end{equation}
\end{itemize}
In particular, these relations imply
\[
F_n = \frac{1}{A}\Bigl(A - \sum_{j=1}^n a_j z_j\Bigr)^2 - \frac{1}{A} \sum_{j=1}^n a_j z_j^2.
\]

Locally, near each edge of $\mathbf{\Delta}$, the Ricci-flat metric extends smoothly to a $4$-manifold if and only if the vector $\overline{\mv}_j$ is primitive in the lattice $\Lambda$, and any two adjacent vectors $\overline{\mv}_j, \overline{\mv}_{j+1}$ generate $\Lambda$.

Based on these strong constraints, \cite{BG} obtained a classification of toric Hermitian $\ALF/\AF$ gravitational instantons:
\begin{theorem}[\cite{BG}]\label{thm:BG}
A toric Hermitian $\ALF/\AF$ gravitational instanton must be, up to scaling, one of the following:
\begin{itemize}
    \item the anti-Taub-NUT space;
    \item the $\AF_\beta$ Kerr space for some $\beta \in [0,1)$;
    \item the $\AF_\beta$ Chen-Teo space for some $\beta \in (0,1)$;
    \item the Taub-Bolt or anti-Taub-Bolt space.
\end{itemize}
\end{theorem}

Moreover, \cite{BG} gives an explicit description of all these metrics in terms of their associated toric polytopes. From left to right, the corresponding (unbounded) polytopes are those of the anti-Taub-NUT, Kerr, Chen-Teo, anti-Taub-Bolt, and Taub-Bolt gravitational instantons.
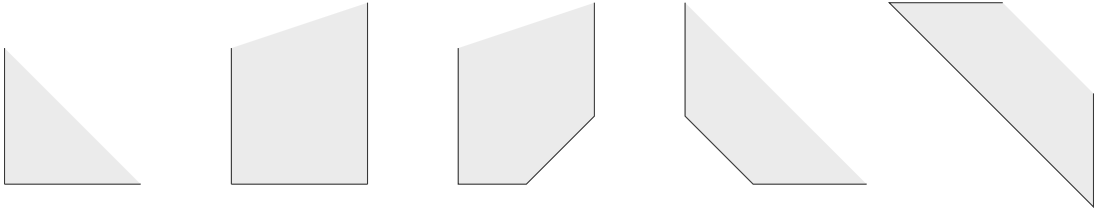
\begin{figure}[ht]
    \begin{tikzpicture}[scale=0.6]
        \draw (0,0)--(3,0);
        \draw (0,0)--(0,3);
        \filldraw[color=gray!50,fill=gray!50,opacity=0.3](3,0)--(0,0) --(0,3);

        \draw (5,3)--(5,0);
        \draw (5,0)--(8,0);
        \draw (8,0)--(8,4);
        \filldraw[color=gray!50,fill=gray!50,opacity=0.3](5,3)--(5,0) --(8,0)--(8,4);

        \draw (10,3)--(10,0)--(11.5,0)--(13,1.5)--(13,4);
        \filldraw[color=gray!50,fill=gray!50,opacity=0.3] (10,3)--(10,0)--(11.5,0)--(13,1.5)--(13,4);

        \draw (15,4)--(15,1.5)--(16.5,0)--(19,0);
        \filldraw[color=gray!50,fill=gray!50,opacity=0.3](15,4)--(15,1.5)--(16.5,0)--(19,0);

        \draw (22,4)--(19.5,4)--(24,-0.5)--(24,2);
        \filldraw[color=gray!50,fill=gray!50,opacity=0.3](22,4)--(19.5,4)--(24,-0.5)--(24,2);
    \end{tikzpicture}
	\caption{The polytopes}
	 \label{Figure:polytopes}   
\end{figure}
There is also related work by \cite{OliveiraSenaDias}, which provides an alternative proof of the above toric classification under weaker assumptions at infinity.

There is a striking contrast between this result and the situation for Hermitian Einstein metrics with positive Einstein constant. For a Type $\mathrm{II}$ Einstein metric with positive Einstein constant, up to a finite cover, the metric is either K\"ahler or Hermitian. It was proved by LeBrun \cite{LeBrunEinsteinHermitian4Manifolds} that all Hermitian positive Einstein $4$-manifolds are, up to scaling, given by the Page metric on $\mathbb{CP}^2 \sharp \overline{\mathbb{CP}}^2$ \cite{Page}, and the Chen-LeBrun-Weber metric on $\mathbb{CP}^2 \sharp 2\overline{\mathbb{CP}}^2$ \cite{ChenLeBrunWeber}. An observation of LeBrun \cite{LeBrunEinsteinMetricsComplexSurfaces} further shows that all such metrics are toric.

Motivated by this comparison with the positive Einstein case, it was conjectured in \cite{AA} that all $\ALF$ and $\AF$ Hermitian gravitational instantons are likewise toric. In particular, this would imply that the list in Theorem~\ref{thm:BG} exhausts all $\ALF$ and $\AF$ Hermitian gravitational instantons. This conjecture was proved in \cite{Li2}; see also the discussion in Section~\ref{ss:Hermitian B}.

\subsection{Classification Results A: Asymptotic Structure}

As in the hyperk\"ahler case, we now have a complete classification of the asymptotic ends of Hermitian gravitational instantons. 

\begin{theorem}[Asymptotic Classification \cite{Li2}]\label{thm:Hermitian A}
    Let $(M, g)$ be a Hermitian non-K\"ahler gravitational instanton, then up to scaling it must be  $\ALE$, $\ALF$ or $\AF$.
\end{theorem}
\begin{remark}
   In \cite{Li2} a stronger classification result is proved: if we only assume $(M, g)$ is a Riemannian end which is Ricci-flat, Hermitian and with finite energy, then we may also see special Kasner or anti-$\ALH^*$ models. On the other hand, these two model ends cannot appear if $(M, g)$ is complete.
\end{remark}
  In the remainder of this subsection, we sketch a proof of the above theorem. In the non-collapsed case, $(M,g)$ must be $\ALE$ by \cite{BKN}, so we henceforth assume that $(M,g)$ is collapsed at infinity. Recall that there is a canonical Killing field $\mathbb{K}$.
  
\

\textbf{Claim.} $\mathbb K$ has no zeroes outside a compact set. 

\

To see this, we first note that the moment map $\xi := \Sc_{g_K}$ is a Morse-Bott function on $M$, whose critical submanifolds have indices $0$, $2$, or $4$. Since $\Sc_{g_K}$ is nowhere vanishing, it follows from \eqref{eqn:scalar curvature equation} and the maximum principle that $\xi$ cannot admit a local minimum. Moreover, it has exactly one maximal critical submanifold, while all other critical submanifolds consist of isolated points of index $2$.

It therefore suffices to show that the index $2$ critical points remain in a compact set. Suppose, for contradiction, that this is not the case. Then there exists a sequence of index $2$ critical points $q_j$ tending to infinity. Let $a_j > 0$ and $-b_j < 0$ denote the associated weights of $\mathbb{K}$ at $q_j$. Fix a base point $p$, and consider the rescaled spaces $(M, g_j, p) := (M, r_j^{-2} g, p)$, where $r_j = \dist_g(p, q_j)$.

For sufficiently small $\delta$, the balls $B_j := B(q_j, \delta) \subset (M, g_j)$ have bounded curvature and non-collapsed universal covers $(\widetilde B_j, \widetilde q_j)$. Passing to a subsequence, we may assume that $(B_j, q_j)$ converges to a limit ball $B_\infty := B(q, \delta)$, and that the universal covers $(\widetilde B_j, \widetilde q_j)$ converge to a limit $(\widetilde B_\infty, \widetilde q)$. The actions of the deck transformation groups $G_j$ on $\widetilde B_j$ converge to an isometric action of a Lie group $G_\infty$ of positive rank on $\widetilde B_\infty$, and
\[
B_\infty = \widetilde B_\infty / G_\infty.
\]

Since $\mathbb{K}$ vanishes at $q_j$, after suitable normalization we may assume that it converges to a nonzero Killing field $\mathbb{K}_\infty$ on $\widetilde B_\infty$ that vanishes at $\widetilde q$. It then follows that $\mathbb{K}_\infty$ must vanish along the orbit $G_\infty \cdot \widetilde q$. Consequently, we must have either $a_j / b_j \to 0$ or $a_j / b_j \to \infty$. In either case, the $\mathbb{R}$-action generated by $\mathbb{K}$ limits to an effective isometric $\mathbb{R}^2$-action on $\widetilde B_\infty$ that fixes the orbit $G_\infty \cdot \widetilde q$, which is impossible.

\

   We now divide into 2 cases 
\begin{itemize}
    \item $\mathbb{K}$ induces an $S^1$-action;
    \item $\mathbb{K}$ induces a higher rank $\mathbb T^k$-action.
\end{itemize}

\subsubsection{$\mathbb{K}$ induces an $S^1$-action}
\label{subsubsec:S1 positive case}

In this case, we may perform the K\"ahler reduction. For simplicity, we assume that the $S^1$-action is free; in general, one must work in the orbifold setting (see \cite{Li2}). Then each level set $\{\xi = \epsilon\}$ is an $S^1$-bundle over a compact Riemann surface $\Sigma$.

In \eqref{eq:lebrun-tod ricci flat metric}, the metric
$$
\underline g := e^u (dx^2 + dy^2)
$$
can be viewed globally as a family of metrics on $\Sigma$ parametrized by $\varrho$. Note that at infinity $\xi = \Sc_{g_K} \to 0$, and hence $\varrho \to \infty$.

Integrating the $SU(\infty)$ Toda equation
$$
(e^u)_{\varrho\varrho} + u_{xx} + u_{yy} = 0
$$
over $\Sigma$, and applying the Gauss-Bonnet formula, we obtain
\begin{equation}\label{eq:e^u volume of sigma}
    \int_{\Sigma} e^u \, dx\,dy = 2\pi \chi(\Sigma)\varrho^2 + a\varrho + b,
\end{equation}
for some constants $a$ and $b$. Using \eqref{eq:V in lebrun-tod}, one also finds
\begin{equation}\label{eq:Ve^u volume of sigma}
    \int_{\Sigma} V e^u \, dx\,dy = -6a\varrho^2 - 12b\varrho.
\end{equation}
Suppose that the period of the vector field $\mathbb{K}$ is $\mathfrak{p}$. Then we obtain the volume formula
\begin{equation}\label{eq:volume of rho annuli}
    \Vol(D_0 \le \varrho \le D_1)
    = -2\mathfrak{p} a (D_1^3 - D_0^3) - 6\mathfrak{p} b (D_1^2 - D_0^2).
\end{equation}
Since $\varrho > 0$, it follows from \eqref{eq:e^u volume of sigma}--\eqref{eq:volume of rho annuli} that $\chi(\Sigma) > 0$, and hence $\Sigma$ is conformally equivalent to $S^2$.

We now consider an asymptotic cone $(\mathcal C, O)$, defined as the pointed Gromov-Hausdorff limit of the rescaled spaces
\[
(M_j, g_j, p) := (M, r_j^{-2} g, p),
\]
for a sequence $r_j \to \infty$. As before, for any $q \in \mathcal C \setminus \{O\}$, a ball $B_\infty$ around $q$ is of the form
\[
B_\infty = \widetilde B_\infty / G_\infty,
\]
where $(\widetilde B_\infty, \tilde g_\infty)$ arises as the limit of the local universal covers. Our goal is to pass the conformally K\"ahler structure to the limit $\widetilde B_\infty$. This is more delicate than in the hyperk\"ahler case, since Type $\mathrm{II}$ metrics may degenerate to Type $\mathrm{I}$ metrics in the limit.

On $M_j$, we define
$$
\xi_j := (\xi(q_j))^{-1}\xi, \qquad \varrho_j := \xi_j^{-1}.
$$
By the scalar curvature formula under conformal change, we obtain
\begin{equation}\label{eq:scaled down lambda_j}
    6\Delta_{g_j}\xi_j + r_j^2 \xi(q_j)^3 \xi_j^4 = 0.
\end{equation}
Noting that $\xi(q_j) \le C r_j^{-2/3}$, the Cheng-Yau gradient estimate yields a Harnack inequality for $\xi_j$. In particular, after appropriate normalization and passing to a subsequence, there exists a smooth \emph{positive} limit function $\xi_\infty$ on $\widetilde B_\infty$ such that the conformal metric $\xi_\infty^2 \widetilde g_\infty$ is K\"ahler.

Moreover, $\xi_\infty$ descends to $B_\infty$ and extends to a globally defined function on $\mathcal C \setminus \{O\}$. From volume comparison and the formula \eqref{eq:volume of rho annuli}, it follows that $\xi_\infty$ is non-constant, with $\xi_\infty \to 0$ at $O$ and $\xi_\infty \to \infty$ at infinity.

After pulling back to each local cover $\widetilde{B}_j$, we may assume that the vector fields $\mathbb{K}_j := -J \nabla_{g_j} \xi_j^{-1}$ and functions $V_j := |\mathbb{K}_j|_{g_j}^{-2}$ converge, respectively, to
\[
\mathbb{K}_\infty := -J \nabla_{g_\infty} \xi_\infty^{-1}, 
\qquad 
V_\infty := |\mathbb{K}_\infty|_{g_\infty}^{-2}
\]
on $\widetilde B_\infty$.
We may rewrite \eqref{eq:lebrun-tod ricci flat metric} as
\begin{equation} \label{eqn5-30}g_j=V_j(d\varrho_j^2+e^{u_j}(dx^2+dy^2))+r_j^{-4}\xi^{-2}(q_j)V_j^{-1}\eta^2,
\end{equation}
where $$V_j=\xi^{-3}(q_j)r_j^{-2}(-12\varrho_j+6\varrho_j^2\partial_{\varrho_j}u_j)$$ and $$u_j=u+\log\xi^2(q_j).$$
We may also assume that $u_j$ converges to a smooth limit $u_\infty$. 
Denote  $$h_{\Sigma,j}:=e^{u_j}(dx^2+dy^2)$$
as the metric on $\Sigma$ determined by $\varrho_j=1$.
Using the above local convergence, one sees that for sufficiently large $j$, the metrics $h_{\Sigma,j}$ have uniformly bounded curvature. 
Furthermore, since $\chi(\Sigma) > 0$, the sequence $(\Sigma, h_{\Sigma,j})$ cannot collapse everywhere.
Without loss of generality, we may choose $q_j$ such that its image $q_{\Sigma,j} \in (\Sigma, h_{\Sigma,j})$ has injectivity radius bounded from below. 
It then follows that, after passing to a subsequence, the balls $B_\epsilon(q_{_{\Sigma,j}})\subset (\Sigma,h_{\Sigma,j})$ converges smoothly to a non-collapsed limit $(B_{\epsilon}(q_{\Sigma,\infty}),h_{\Sigma,\infty})$.

After passing to a subsequence, we may assume that one of the following occurs:
\begin{itemize}
    \item $\xi_j^3 r_j^2$ converges to a positive constant. Geometrically, this means that the Ricci-flat metric $(\widetilde B_\infty, \widetilde g_\infty)$ is of Type $\II$.
    \item $\xi_j^3 r_j^2$ converges to zero. Geometrically, this means that $(\widetilde B_\infty, \widetilde g_\infty)$ is of Type $\I$.
\end{itemize}
However, one can show that the first situation cannot occur, by considering the volume of the metric annuli $A_{1/2,2}$ in $(M,g_j,p)$ and using the choice of $q_j$, \eqref{eq:volume of rho annuli}, and \eqref{eqn5-30}.

In the second case,  the K\"ahler metric $(\widetilde B_\infty,\lambda^{2/3}_\infty\widetilde{g}_\infty)$ has vanishing scalar curvature, and we have $$-12\varrho_{\infty}+6\varrho_\infty^2\partial_{\varrho_\infty}u_\infty\equiv0. $$ It then follows that $h_{\Sigma,\infty}$ has constant curvature one. This implies that $h_{\Sigma_j}$ has bounded diameter and converges to the round metric on $S^2$. The asymptotic cone $\mathcal C$ is given by the metric 
$$g_\infty=V_\infty(d\varrho_\infty^2+\varrho_\infty^2g_{S^2}).$$
on $\mathbb R^3$.

A further argument based on LeBrun's ansatz for scalar-flat K\"ahler metrics~(\cite{LeBrun91}) shows that $V_\infty$ is harmonic with respect to the flat metric
\[
d\varrho_\infty^2 + \varrho_\infty^2 g_{S^2}.
\]
It then follows from \eqref{eq:e^u volume of sigma}--\eqref{eq:Ve^u volume of sigma}, together with the non-maximal volume growth assumption, that $V_\infty$ must be constant. This implies that $\mathcal C$ is itself isometric to $\mathbb{R}^3$.

In general, when the $S^1$-action is not necessarily free, a similar argument shows that $\mathcal C$ must be a quotient of the form $\mathbb{R}^3 / \mathbb{Z}_k$, where $\mathbb{Z}_k$ acts on an $\mathbb{R}^2$-factor. In particular, the asymptotic cone is unique.

From this, together with a PDE analysis of \eqref{eq:toda equation}--\eqref{eq:V in lebrun-tod}, one obtains an asymptotic expansion for $u$ and $V$. This implies that $(M, g)$ is up to scaling either $\ALF$ or $\AF_\beta$ for some $\beta \in \mathbb{Q}$.

\subsubsection{$\mathbb{K}$ induces a higher rank $\mathbb{T}^k$-action}
If $k \geq 3$, then one can explicitly write down solutions to the LeBrun-Tod ansatz and verify that such solutions cannot arise when $M$ is complete. Thus, we may assume $k = 2$. In this case, the unique maximal submanifold of $\xi = \Sc_{g_K}$ must be a fixed point. In particular, this implies that $M$ is simply connected.

Since the $T^2$-action has a fixed point, it follows that there exists a global moment map
\[
(x_1, x_2) : M \to \mathbb{R}^2.
\]
In this setting, although a global K\"ahler reduction with respect to the vector field $\mathbb K$ is not available, the local analysis still applies and allows one to study the structure of an asymptotic cone $(\mathcal C, O)$.
As before, away from $O$, the space $\mathcal C$ is locally given as a quotient of the form $(\widetilde B_\infty, \widetilde g_\infty)/G_\infty$. One can show that $\widetilde g_\infty$ must be of Type $\I$. Moreover, $\mathcal C \setminus \{O\}$ is a two-dimensional manifold with boundary, equipped with a Riemannian metric $g_\infty$, such that the conformally related metric $V_\infty^{-1} g_\infty$ is flat. 

Using arguments similar to those in \cite{LS2024}, one concludes that the flat space is a half-plane $\mathbb{H}$. Since $V_\infty$ is harmonic, it must be constant. Therefore, $\mathcal C$ itself is the flat half-plane $\mathbb{H}$, and the local limit $(\widetilde B_\infty, \widetilde g_\infty)$ is also flat.

The above discussion implies that $(M,g)$ is asymptotically flat, i.e., $|\Rm| = o(r^{-2})$. By the resolution of a conjecture of Petrunin-Tuschmann in \cite{LS2024}, it follows that the end of $M$ cannot be simply connected, and hence must be a finite quotient of $S^2 \times S^1 \times [1,\infty)$.
Let $K \subset M$ be a compact set, and consider the universal cover $\widetilde M$ of $M \setminus K$. One can then perform a global K\"ahler reduction on $\widetilde M$, which leads to the conclusion that $(M,g)$ is $\AF_\beta$ for some $\beta \notin \mathbb Q$.

\subsection{Classification Results B: Global Structure}
\label{ss:Hermitian B}

The following result provides a complete classification of Hermitian gravitational instantons in the collapsed case, and a partial classification in the non-collapsed case. 
\begin{theorem}[Global classification \cite{Li1,Li2}]\label{thm:Hermitian B}
    Let $(M,g)$ be a Hermitian  gravitational instanton.  Then, up to scaling, $(M, g)$ must be of the following
    \begin{itemize}
        \item the anti-Calabi-Eguchi-Hanson gravitational instanton, if $g$ is $\ALE$ with structure group $\Gamma\subset SU(2)$;
        \item  one of the spaces given in Theorem \ref{thm:BG}, if $g$ is $\ALF/\AF$. 
    \end{itemize}
\end{theorem}
\begin{remark}
In the $\ALF$ and $\AF$ cases, the above confirms a conjecture of \cite{AA}. There is also related work in \cite{BernardoLucietti}, following the approach of \cite{BG}, which shows that in the $\ALE$ case, under a toric assumption, the only possibility is the anti-Calabi-Eguchi-Hanson gravitational instanton.
\end{remark}

In the remainder of this subsection, we sketch a proof of the above theorem. We begin with the $\ALF$ and $\AF$ cases. The strategy is to show that such a gravitational instanton must be toric, and then to appeal to Theorem~\ref{thm:BG}.

Let $(g_K, J)$ denote the conformally K\"ahler metric. If the Killing field $\mathbb{K}$ generates a $\mathbb{T}^2$-action, then the metric is already toric. Thus, we may assume that $\mathbb{K}$ generates only an $S^1$-action. In this case, $(M,g)$ is either $\ALF$ or $\AF_\beta$ with $\beta \in \mathbb{Q}$. By the asymptotic classification (Theorem~\ref{thm:Hermitian A}), the K\"ahler metric $g_K$ exhibits cusp behavior at infinity, in the sense that locally it is asymptotic to a Poincar\'e cusp bundle over a round $\mathbb{S}^2$.

There are two main steps:
\begin{enumerate}
    \item Show that $(M, J)$ admits an effective holomorphic $(\mathbb{C}^*)^2$-action extending the $S^1$-action generated by $\mathbb K$;
    \item Prove a Calabi-type theorem, producing a compact torus $\mathbb T^2 \subset (\mathbb{C}^*)^2$ that preserves $g_K$.
\end{enumerate}

Step (1) follows from a compactification result. Denote by $h$ the Hermitian metric on $K_M^{-1}$ induced by $g_K$. An observation of LeBrun~\cite{LeBrunEinsteinMetricsComplexSurfaces} implies that the Hermitian metric $\Sc_{g_K}^{-2} h$ on $K_M^{-1}$ has positive curvature. 

One can naturally compactify $(M, J)$ to a two-dimensional complex orbifold $\overline{M}$ by adding a point to each $\mathbb{C}^*$-orbit at infinity. The complement $\overline M \setminus M$ is a divisor $D = \mathbb{CP}^1$. The $\mathbb{C}^*$-action induced by $\mathbb K$ extends to $\overline{M}$ and fixes $D$.

By studying the positive curvature $2$-form of the Hermitian metric $\Sc_{g_K}^{-2} h$, one shows that the line bundle $-(K_{\overline M} + D)$ is positive on $\overline{M}$. Hence $(\overline{M}, D)$ is a log del Pezzo surface. By performing suitable blow-ups, one can further show that $\overline{M}$ is birational to a ruled surface over $\mathbb{CP}^1$, and hence is rational. Taking a minimal resolution of $\overline{M}$ with the lifted holomorphic $\mathbb{C}^*$-action, and then passing to a minimal model $M'$, we obtain a surface that still admits a holomorphic $\mathbb{C}^*$-action. It follows that $M'$ must be either $\mathbb{CP}^2$ or a Hirzebruch surface. Under these strong constraints, one concludes that $\overline{M}$ admits a holomorphic toric $(\mathbb{C}^*)^2$-action containing the $S^1$-action generated by $\mathbb K$.

For Step (2), recall that for a compact extremal K\"ahler metric, Calabi~\cite{Calabi1970RicciII} proved that the isometry group contains a maximal compact subgroup of the identity component of the holomorphic automorphism group. In our setting, one can extend this result using the fact that $g_K$ has Poincar\'e-type behavior at infinity.

\

Finally, we consider the $\ALE$ case. In contrast to the $\ALF$ and $\AF$ cases, the K\"ahler metric $g_K$ is incomplete on $M$, and its completion $\overline{M}$ is obtained by adding an orbifold point at infinity with orbifold group $\Gamma \subset U(2)$. 
The compactified surface $\overline{M}$ is still a log del Pezzo surface, but in this case one cannot conclude that $\overline{M}$ is toric. The metric $g_K$ extends to an orbifold K\"ahler metric $\overline g_K$ on $\overline{M}$, which is Bach-flat. Its scalar curvature $s_{\overline g_K}$ vanishes at the orbifold point and is positive elsewhere.

In general, for $\Gamma \subset U(2)$, there exist infinitely many log del Pezzo surfaces with a single orbifold point admitting a holomorphic $\mathbb{C}^*$-action fixing that point. However, under the assumption $\Gamma \subset SU(2)$, there are only finitely many such orbifolds. One can compute the minimum value of $\Sc_{g_K}$ in terms of the K\"ahler class and the holomorphic extremal vector field, as essentially carried out by LeBrun-Simanca \cite{LeBrunSimanca}. 

In our setting, one can explicitly verify that the minimum of $\Sc_{g_K}$ is nonzero, except in the case of the orbifold obtained from $\mathbb{P}(\mathcal{O}(2)\oplus \mathcal{O})$ by contracting the $(-2)$-curve, with the $\mathbb{C}^*$-action given by rotation along the fibers. This exceptional case corresponds to the anti-Calabi-Eguchi-Hanson gravitational instanton.

\section{Type \texorpdfstring{$\III$}{III} gravitational instantons}\label{s:Type III}
The study of Type $\III$ gravitational instantons is largely an open area. On the one hand, there are not many known topological or geometric constraints; on the other hand, it is also challenging to solve the PDE system for general Einstein metrics. To date there are only two classes of examples of Type $\III$ (for both orientations) gravitational instantons are known, by \cite{KRWY} and  \cite{LiSun2025}.  In this section we will explain these results.

\subsection{Dimension reduction}
\label{subsec:reduction}

Let $(M,g)$ be an oriented  4-dimensional Riemannian manifold. Suppose $X$ is a Killing vector field on $M$ and denote by $\eta=X^{\flat}$ its dual 1-form. Then one can check that 
$$d^*(\eta\wedge d\eta)=2(\Ric_g.\eta) \wedge \eta. $$
In particular, if $g$ is Ricci-flat, then locally one can find an \emph{twist potential}
 $\mathcal E$, which is unique up to a constant addition,  such that  $$d\mathcal{E}=*(\eta\wedge d\eta).$$ 
The \emph{twist 1-form} $d\mathcal{E}$ measures the non-integrability of the orthogonal distribution to $X$. 
Clearly 
$$\mathcal L_{X}\mathcal E=\langle \eta, d\mathcal E\rangle=*(\eta\wedge *d\mathcal E)=0.$$
The twist was first defined in the context of general relativity for a stationary vacuum spacetime, i.e, a Lorentzian Einstein metric with a time-like Killing field. 

Now suppose we have two commuting Killing fields $X_1, X_2$, with dual 1-forms $\eta_1, \eta_2$ respectively.  Then we get two twist potentials $\mathcal{E}_1, \mathcal{E}_2$ respectively.  
Denote $$\Theta_i:=*(\eta_1\wedge\eta_2\wedge d\eta_i),\ i=1, 2.$$ Then 
$$d\Theta_1=d\langle\eta_2, d\mathcal{E}_1\rangle=d\mathcal L_{X_2}\mathcal{E}_1=\mathcal L_{X_2} \eta_1=0,$$ 
and similarly, $d\Theta_2=0$. So $\Theta_1$ and $\Theta_2$ are both constants. They are zero precisely when the distribution orthogonal to the span of $X_1, X_2$ is integrable. This is the case if $X_1$ and $X_2$ have a common zero. 

The above discussion motivates one to work with Riemannian metrics $g$ on $B^2\times \mathbb T^2$ which are $\mathbb T^2$-invariant. Assuming the orthogonal distribution is integrable, they are locally  of the form 
\begin{equation}
    g=\underline g+\rho \sum_{1\leq\alpha, \beta\leq 2}\Phi_{\alpha\beta}d\phi_\alpha d\phi_\beta
\end{equation}
where $\underline g$ is a metric on $B$, $\rho>0$ is a function on $B$, and $\Phi$ is a smooth map from $B$ into $\mathrm{SL}(2;\R)/\SO(2)$; the latter is identified with the space of positive definite symmetric 2 by 2 matrices with determinant 1. We also normalize the area form such that 
$$\int_{\mathbb T^2} d\phi_1d\phi_2=1.$$
To derive the dimension reduction of the Ricci-flat equation in this setting, we recall that a Ricci-flat metric is a critical point of the Einstein-Hilbert functional 
$$\EH(g)=\int_{M} \Sc_g \dvol_g.$$
A straightforward computation shows that 
\begin{equation}
    \Sc_g=\Sc_{\underline g}-2\rho^{-1}\Delta_{\underline g}\rho+\frac{1}{2}\rho^{-2}|d\rho|^2_{\underline g}-\frac{1}{4}\Tr(\Phi^{-1}d\Phi )_{\underline g}^2. 
\end{equation}
So we may write 
\begin{equation}
    \EH(g)=\int_B (\Sc_{\underline g}\rho-2\Delta_{\underline g}\rho+\frac{1}{2}\rho^{-1}|d\rho|^2_{\underline g}-\frac{1}{4}\rho\Tr(\Phi^{-1}d\Phi )_{\underline g}^2)\dvol_{\underline g}.
\end{equation}
To derive the Euler-Lagrange equation of $\EH$, we work at a critical point of $\EH$ and  consider compactly supported variations of $\rho, \Phi$ and $\underline{g}$ separately. 
First, if we conformally vary $\underline g$ to $(1+u)\underline g$, then noticing that the last three integrals are conformally invariant, so 
$$\delta\EH=\int_B -\frac{1}{2}\rho \Delta_{\underline g}u \dvol_{\underline g}=-\frac{1}{2}\int_B u\Delta_{\underline g} \rho \dvol_{\underline g} .  $$
This implies that $\rho$ is a harmonic function on $B$. Let $z$ be its harmonic dual, i.e., $dz=*d\rho$. Generically we may assume $\rho+\sqrt{-1}z$ gives local \emph{conformal} coordinates for $\underline g$. 

If we consider the variation of $\Phi$, then we get 
	\begin{equation}\label{eqn:axisymmetric}
	    (\rho\Phi^{-1}\Phi_\rho)_\rho+(\rho\Phi^{-1}\Phi_z)_z=0.
	\end{equation}	
This is the same as saying that $\Phi$ defines an axisymmetric harmonic map in 3 dimensions, if we use cylindrical coordinates $(\rho, z, \theta)$. 
    
If we vary the conformal class of $\underline g$, i.e., vary $\underline g$ by $h$ with $\tr_{\underline g}h=0$, then we have
\begin{equation}\label{eqn6-4}\nabla^2\rho+\frac{1}{2}\rho^{-1}d\rho\otimes d\rho -\frac{\rho}{4}\Tr(d\Phi\Phi^{-1})^2=c \underline g
\end{equation}
for some constant $c$.
Finally, if we vary $\rho$, we get 
\begin{equation}\label{eqn6-5}
    \Sc_{\underline g}=\Delta_{\underline g}\log\rho+\frac{1}{2}\rho^{-2}|d\rho|^2_{\underline g}+\frac{1}{4}\Tr(\Phi^{-1}d\Phi)_{\underline g}^2.
\end{equation}
One can see that \eqref{eqn6-4} and \eqref{eqn6-5} combine to one equation 
\begin{equation}
    \Ric_{\underline g}=\nabla^2\log\rho+\frac{1}{4}\Tr(\Phi^{-1}d\Phi+\rho^{-1}d\rho \Id)^2_{\underline{g}}.
\end{equation}
If we write the metric 
$$\underline g=e^{2\nu}(d\rho^2+dz^2),$$
then this is equivalent to the system 
\begin{align} \label{e:eqn2.33}
	\left\{\begin{aligned}
		&\p_z\nu=\frac{1}{4\rho}\Tr(\mathcal U\mathcal V),\\
		&\p_\rho\nu=-\frac{1}{2\rho}+\frac{1}{8\rho}\Tr(\mathcal U^2-\mathcal V^2),
	\end{aligned}\right.
\end{align}
where we denote  
$$\mathcal U=\Id+\rho\Phi^{-1}\Phi_\rho,$$ 
and 
$$\mathcal V=\rho\Phi^{-1}\Phi_z.$$ 
The compatibility between the two equations in \eqref{e:eqn2.33} follows from \eqref{eqn:axisymmetric}. It also implies that $\nu$ is determined up to constant by $\Phi$.

The upshot is that in the above setting, a Ricci-flat metric locally reduces to an \emph{augmented harmonic map}, which consists of an axisymmetric harmonic map $\Phi$ together with a choice of the \emph{augmentation} $\nu$. Such a formalism is well-studied in the general relativity literature, see \cite{Kunduri, Lott, Harmark}.

\subsection{Rod structures}
We are interested in \emph{toric} gravitational instantons, namely, gravitational instantons with an effective action of a torus $\mathbb T^2$. The topology is encoded in a \emph{rod structure}, which plays the role of a Delzant polytope in toric K\"ahler geometry. For a rigorous and precise definition of a rod structure and related terminologies we refer to \cite[Section 3]{LiSun2025}. For the expository purpose in this paper,  we give a less technical and slightly different definition from \cite{LiSun2025}.

\begin{definition}A rod structure $\mathfrak R$ consists of the following data
\begin{itemize}
  \item   A collection of \emph{turning points} in $\R$ \[ -\infty< z_1 = 0 < \dots < z_n <+\infty.\]
  \item A lattice $\Lambda\subset\R^2$.
    \item A  primitive \emph{rod vector}  $\mathfrak v_j\in \Lambda$ for each rod $$\mathfrak I_j:=(z_j, z_{j+1}),  \quad j=0, \cdots, n,$$ such that $\Lambda$ is generated by $\mathfrak v_j$ and $\mathfrak v_{j+1}$. Here we denote $z_0=-\infty$ and $z_{n+1}=+\infty$. 
    If $\mv_0\neq\mv_n$, then we require that $\Lambda$ is generated by $\mv_0$ and $\mv_n$ and
$|\mv_0|=|\mv_n|;
  $ 
     if $\mv_0=\mv_n$, then we require that $\Lambda$ is generated by $\mv_0$ and $\mv_0'$, where $\mv_0'$ is obtained by rotating $\mv_0$ counterclockwise by $\pi/2$.
    \item An asymptotic \emph{Type} $\sharp$, which is either $\Ae, \ALF^+, \ALF^-$ if $\mv_0\neq \pm\mv_n$, and $\AFa$ for some $\beta\in \R$ if $\mv_0=\pm\mv_n$
    \end{itemize}
    \end{definition}

\begin{remark}
    Here $\Ae$ means $\ALE$ with the finite group $\Gamma=\{\Id\}$. 
\end{remark}
    
    Given a rod structure $\mathfrak R$, we can build an oriented smooth toric four manifold (without specifying a Riemannian metric) as follows. In the $(\rho, z)$ plane we denote
    $$\mathbb H:=\{\rho\geq 0\}. $$
    Denote by $\mathbb H^\circ$ the interior of $\mathbb H$ and we identify $\p\mathbb H$ with $\R$ so the turning points and the rods are naturally subsets of $\p\mathbb H$.
Consider the product 
$$M^\circ:=\mathbb H^\circ\times \mathbb T^2,$$
where $\mathbb T^2:=\mathbb R^2/\Lambda$. We endow $M^\circ$ with the canonical orientation induced by the 4-form $d\rho\wedge dz\wedge d\phi_1\wedge d\phi_2$, where $\phi_1, \phi_2$ are the natural coordinates on $\R^2$. It admits a natural translation action of the 2-torus $\mathbb T^2$. The open  manifold $M^\circ$ can be partially compactified into a smooth toric manifold $M$, by adding points with stabilizers over $\p\mathbb H$: over each rod $\mathfrak I_j$ we add a copy of $\dS^1$, which is the quotient of $\mathbb T^2$ by the subgroup generated by the rod vector $\mathfrak v_j$, and at each turning point $z_j$ we  add a point whose stabilizer is the whole $\mathbb T^2$.

Changing the orientation of this $4$-manifold corresponds to replacing $\mathfrak R$ by its \emph{opposite} $\overline{\mathfrak R}$. The opposite rod structure has the same turning points as $\mathfrak R$, but the rod vectors $\mathfrak v_j$ are replaced by $R_1\mathfrak v_j$, and the asymptotic Type $\sharp$ is replaced by $\overline{\sharp}$. Here $R_1$ denotes the reflection in $\mathbb R^2$ that swaps the two coordinate axes, and the Type transforms according to
\[
\overline{\Ae}=\Ae, \qquad 
\overline{\ALF^{\pm}}=\ALF^{\mp}, \qquad 
\overline{\AFa}=\AF_{-\beta}.
\]

Given a rod structure $\mathfrak R$, we may associate to it a nonnegative integer \emph{degree} $d(\mathfrak R)\ge 0$. Starting from the line $[\mathfrak v_0]$, for each $j=1,\ldots,n$ we rotate the line $[\mathfrak v_{j-1}]$ counterclockwise by an angle less than $\pi$ to obtain the line $[\mathfrak v_j]$. If $\mathfrak v_n=\mathfrak v_0$, this procedure produces a loop in $\R\mathbb P^1$. If $\mathfrak v_n\neq \mathfrak v_0$, we perform one additional counterclockwise rotation from $[\mathfrak v_n]$ to $[\mathfrak v_0]$ to close the loop. We then define $d(\mathfrak R)$ to be the winding number of this loop, with a shift by $-1$ when $\mathfrak R$ is of Type $\Ae$ in the $\ALF^+$ case. This shift is not essential and is introduced only to maintain uniform notation. In general, we have $d(\overline{\mathfrak R})\neq d(\mathfrak R)$. One can see that for a  rod structure $\mR$ with asymptotic Type $\Ae/\ALF/\AFa$, if $\mR$ has more than 2 turning points, then $d(\mR)>1$.

\subsection{Model examples}\label{ss:Type III model examples}
Examples of toric gravitational instantons can be obtained by applying the Gibbons-Hawking ansatz  to an axisymmetric harmonic function on $\R^3$. The associated axisymmetric harmonic maps also serve as local and asymptotic models for solving general axisymmetric harmonic map equation; see Section \ref{subsec:tame harmonic maps}. 

Denote the flat metric on $\R^3$ by $d\rho^2+dz^2+\rho^2d\phi_2^2$. For an axisymmetric harmonic function $V$,  we can locally write the Gibbons-Hawking ansatz as 
\begin{equation}\label{eq:Type I metric}
    g=V(d\rho^2+dz^2+\rho^2d\phi_2^2)+V^{-1}(d\phi_1+Fd\phi_2)^2
\end{equation}
with $$dF=\rho V_\rho dz-\rho V_zd\rho.$$ Then the associated axisymmetric harmonic map $\Phi$ is given by 
\begin{equation}\label{eqn:GH harmonic map}
    \Phi=\rho^{-1}\left(\begin{matrix}
        V^{-1} & V^{-1}F \\ V^{-1}F & \rho^2 V+V^{-1}F^2
    \end{matrix}\right).
\end{equation}
The reversed orientation can be obtained by swapping the $\phi_1$ and $\phi_2$ coordinates, which means that the associated axisymmetric harmonic map is given by 
\begin{equation}\label{eqn:GH harmonic map with reversed orientation}
    \Phi=\rho^{-1}\left(\begin{matrix}
      \rho^2 V+V^{-1}F^2 & V^{-1}F \\ V^{-1}F &   V^{-1}
    \end{matrix}\right).
\end{equation}

As more specific examples,  the flat $\R^4$ can be obtained by applying \eqref{eqn:GH harmonic map} to 
$$V=\frac{1}{2r}, \ \  F=\frac{z}{2r}. $$ 
and the rod structure has exactly one turning point $z_0=0$, and with rod vectors  $$\mv_0=(\frac{1}{2}, 1)^{\top}, \ \ \mv_1=(-\frac{1}{2}, 1)^{\top}.$$
For the Taub-NUT metric we have the same rod structure but with $$V=1+\frac{1}{2r}, \ \  F=\frac{z}{2r}. $$

If we now let $V=1$ and $F=0$ we get the flat product $\R^3\times \dS^1$, and the associated  harmonic map is 
\begin{equation}\label{e:standard harmonic map AF}\Phi^{\AF_0}=\begin{pmatrix}
\rho^{-1} & 0 \\
0 & \rho
\end{pmatrix}.
\end{equation}
The rod structure has no turning points and the rod vector is $\mv_0=(0, 1)^{\top}$. For $\beta\in \mathbb R$ we may deform $\Phi^{\AF_0}$ to 
\begin{equation}
	\Phi^{\AFa}=\frac{1}{\rho}\left(\begin{matrix}\rho^2&\beta\rho^2\\\beta\rho^2&\beta^2\rho^2+1\end{matrix}\right).
\end{equation}
This gives rise to the  quotient $M_\beta:=(\R^3\times \R)/\langle q_\beta\rangle$,  where $q_\beta$ acts as a rotation by angle $2\pi\beta$ on $\R^3$ and a translation by distance $1$ on $\R^1$.

\subsection{Tame harmonic maps}
\label{subsec:tame harmonic maps}

In order to construct toric gravitational instantons using axisymmetric harmonic maps, we need to prescribe the asymptotics of harmonic maps near the rods and at infinity. This leads to the notion of \emph{tameness} for axisymmetric harmonic maps defined on $\mathbb H^\circ$. We denote 
$$r=\sqrt{\rho^2+z^2}.$$
For a non-zero $v\in \mathbb R^2$, there is a model harmonic map $\Phi_v$, which is given by 
$$\Phi_{v}:=P_v^{\top}\left(\begin{matrix}
			     \rho&0\\0&\rho^{-1}
		    \end{matrix}\right) P_v, $$
            where $P_v$ is an element in $SO(2)$ which rotates $v$ to the direction of $[(1, 0)^{\top}]$. For $\beta\in \R$ we also define  \begin{equation}
\Phi^{\AFa}_{v}=\rho^{-1}P_v^{\top}\left(\begin{matrix}\rho^2&\beta\rho^2\\\beta\rho^2&\beta^2\rho^2+1\end{matrix}\right)P_v.
\end{equation}
            When $[v]=[(1, 0)^{\top}]$, this is the harmonic map associated to the  flat space $M_\beta$.

            Similarly, for two linearly independent vectors $v_1, v_2\in \R^2$, there are 3  model axisymmetric harmonic maps $\Phi^{\Ae}_{v_1, v_2}, \Phi^{\ALFp}_{v_1, v_2}, \Phi^{\ALFm}_{v_1, v_2}$. When $[v_1]=[(\frac{\alpha}{2}, 1)^{\top}]$ and $[v_2]=[(-\frac{\alpha}{2}, 1)^{\top}]$ for $\alpha>0$, we have 
            	 \begin{equation}\label{e:turning point model}
	\Phi^{\Ae}_{v_1, v_2}:=\frac{1}{\rho}\left(\begin{matrix}
		\frac{2r}{\alpha} & z\\
	z & \frac{\alpha r}{2}
	\end{matrix}\right),
\end{equation}
and $\Phi^{\ALFp}$ and $\Phi^{\ALFm}$ are given by the expression in \eqref{eqn:GH harmonic map} and \eqref{eqn:GH harmonic map with reversed orientation} respectively, with $V=1+\frac{\alpha}{2r}$ and $F=\frac{\alpha z}{2r}$.

When $\alpha=1$, these are the harmonic maps associated to the flat $\R^4$, the Taub-NUT metric and the anti-Taub-NUT metric respectively.  For more general $v_1, v_2$, there is a similar definition by reducing to the above case using a rotation in $SO(2)$.

\begin{definition}[Tameness]
    Given a rod structure $\mR$ with asymptotic Type $\sharp$, a smooth map $\Phi: \mathbb H^\circ\rightarrow SL(2;\R)/SO(2)$ is said to be  (globally) \emph{tamed} by $\mR$ if 
\begin{itemize}
	\item For any rod $\mI_j(j=0, \cdots, n)$ and $z\in \mJ_j$, there is a neighborhood $V$ of $z$ in $\mathbb H$ such that $$\sup_{V\cap \mathbb H ^\circ}d(\Phi,\Phi_{\mv_j})<\infty.$$
	\item For each turning point $z_j(j=1, \cdots, n)$,  there is a neighborhood $V$ of $z_j$ in $\mathbb H$ such that   $$\sup_{V\cap \mathbb H^\circ}d(\Phi,T_{z_j}^*\Phi_{\mv_{j}, \mv_{j+1}})<\infty,$$
    here $T_{\underline{z}}$ refers to the translation $$T_{\underline{z}}:\bH\to\bH ;(\rho,z)\mapsto (\rho,z-\underline{z}).$$
    \item There is a compact set $K\subset \mathbb H$, such that outside $K$ $$d(\Phi,\Phi^{\sharp}_{\mv_0})\leq C.$$
\end{itemize}
We say $\Phi$ is \emph{strongly tamed} by $\mR$ if furthermore, 
$$d(\Phi,\Phi^{\sharp}_{\mv_0})=O(r^{-1}),  \ \ r\rightarrow\infty.$$
\end{definition}

We have a general existence result, which is essentially due to  \cite{Weinstein}, \cite{Kunduri}. The key reason is due to the fact that the source space $\R^3$ is flat and the target space $SL(2;\R)/SO(2)$ is non-positively curved.  
\begin{theorem}[\cite{LiSun2025}, Theorem 4.24]\label{t:existence result}	
Given a rod structure $\mathfrak{R}$,  there is a unique axisymmetric harmonic map $\Phi: \dH^\circ\rightarrow SL(2;\R)/SO(2)$ which is strongly tamed by $\mR$. 
\end{theorem}

\subsection{From tame harmonic maps to conical gravitational instantons}

Given a harmonic map $\Phi$ that is strongly tamed by a rod structure $\mR$ of asymptotic Type $\sharp$, we may, by the discussion in Section~\ref{subsec:reduction}, choose an \emph{augmentation} $\nu : \mathbb H^\circ \to \mathbb R$, which is unique up to an additive constant. As explained in Section~\ref{subsec:reduction}, there is a smooth oriented toric $4$-manifold $M$ associated with $\mR$. The augmented harmonic map $(\Phi,\nu)$ then determines a toric Ricci-flat metric $g$ on the open subset $M^\circ$.

\begin{theorem}[\cite{LiSun2025}, Section 4]\label{thm:from tame harmonic map to toric GI}
There exists a unique choice of augmentation $\nu$ such that the resulting metric $g$ is a toric gravitational instanton of $\sharp$ asymptotics,  possibly with cone singularities along the union of $2$-spheres corresponding to the finite rods of $\mR$.
\end{theorem}

The proof of Theorem \ref{thm:from tame harmonic map to toric GI} involves three main  ingredients. We fix a choice of $\nu$. 

First, there is a local regularity theorem along the interior of each rod $\mathfrak I_j$. This relies on analytic estimates for axisymmetric harmonic maps developed in \cite{LiTian, Nguyen, HKWX}. These estimates were originally proved in the context of general relativity to handle singular behavior near rods, and in our setting they imply the following: for each $j = 0, \cdots, n$, there exists a constant $c_j$ such that the metric $g_j$ associated to the augmented harmonic map $(\Phi, \nu + c_j)$ extends as a Lipschitz metric across the $2$-sphere $S_j \subset M$ corresponding to the rod $\mathfrak I_j$.

Using the ellipticity of the Einstein equation together with harmonic coordinates, one can further show that $g_j$ is in fact smooth across $S_j$. In particular, the original metric $g$ may have cone singularities along $S_j$, with cone angle
\begin{equation}
 \vartheta_j = 2\pi e^{c_j}.
\end{equation}
Intrinsically, this can be expressed as
\begin{equation}
 \vartheta_j = 2\pi \lim_{\rho \to 0} \frac{1}{e^\nu \rho} \sqrt{\rho \, \Phi(\mv_j, \mv_j)}.
\end{equation}

Note that these cone singularities are analytically mild: they arise from an incorrect normalization of the rod vector $\mv_j$. If one rescales $\mv_j$ to $e^{-c_j}\mv_j$, then the metric becomes smooth. In particular, the curvature of $g$ remains locally bounded.

Secondly, there is a singularity removal theorem at each turning point $z_j$. To analyze this, consider the two adjacent rods $\mI_{j-1}$ and $\mI_j$. The above discussion shows that $g$ has cone angle $\vartheta_{j-1}$ along $\mI_{j-1}$ and cone angle $\vartheta_j$ along $\mI_j$. 
In a neighborhood of $z_j$, we may replace the lattice $\Lambda$ by a new lattice $\Lambda'$, generated by $2\pi \vartheta_{j-1}^{-1}\mv_{j-1}$ and $2\pi \vartheta_j^{-1}\mv_j$. With respect to $\Lambda'$, the metric $g$ becomes smooth along the corresponding $2$-spheres $S_{j-1}$ and $S_j$. One can then show that $g$ also extends smoothly across the point associated with the turning point $z_j$. This again relies on local elliptic PDE analysis and the removable singularity theorem for Einstein metrics due to Bando-Kasue-Nakajima \cite{BKN}.

Finally, there is a unique choice of $\nu$ such that $\vartheta_0 = 2\pi$. By the definition of asymptotic models and the normalization conditions on $\mv_0$ and $\mv_n$ in the definition of a rod structure, it follows that $\vartheta_n = 2\pi$ as well. Hence, $g$ has no cone singularities outside a compact set.

\medskip

There is also a rigidity result for globally tame harmonic maps. Note that there is a natural action of $SL(2;\mathbb{R})$ on the set of rod structures, given by simultaneously rotating all rod vectors.

\begin{proposition}[\cite{LiSun2025}, Propositions 4.36 and 4.38] \label{p:rigidity}
Suppose $\Phi$ is a harmonic map globally tamed by a rod structure $\mR$. Then there exists $Q \in SL(2;\mathbb{R})$, preserving the directions of $\mv_0$ and $\mv_n$, such that $Q \cdot \Phi$ is the unique harmonic map strongly tamed by $Q \cdot \mR$ (as constructed in Theorem~\ref{t:existence result}).
\end{proposition}

Finally, the following theorem characterizes the Type of $g$ in terms of the degree of the rod structure.

\begin{theorem}[\cite{LiSun2025}, Theorem 7.4]\label{thm:equivalence of degree and type}
Given a rod structure $\mR$, let $g$ be the conical gravitational instanton produced by Theorems~\ref{t:existence result} and \ref{thm:from tame harmonic map to toric GI}. Then:
\begin{itemize}
    \item $g$ is of Type $\mathrm{I}$ if and only if $d(\mathfrak{R}) = 0$;
    \item $g$ is of Type $\mathrm{II}$ if and only if $d(\mathfrak{R}) = 1$;
    \item $g$ is of Type $\mathrm{III}$ if and only if $d(\mathfrak{R}) \ge 2$.
\end{itemize}
\end{theorem}
\begin{remark}
    A more general statement is proved in \cite{LiSun2025}, where the data of the lattice $\Lambda$ is dropped in the definition of a rod structure. Then it gives rise to a PDE classification result for axisymmetric harmonic maps with given asymptotic controls.  
\end{remark}

In both the Type $\mathrm{I}$ and Type $\mathrm{II}$ cases, toric gravitational instantons (possibly with conical singularities) reduce to axisymmetric harmonic functions via the Gibbons-Hawking and LeBrun-Tod ansatz. In particular, they can be classified and admit explicit expressions. In the Type $\I$ case this is easy to write down, see Section \ref{ss:Type III model examples};  in the Type $\mathrm{II}$ case, this was carried out in \cite{BG}, see Section \ref{ss:BG}.

\subsection{Cone angles and rod lengths}
\label{ss:cone angles and rod lengths}

Given the above discussion, it becomes clear that, in order to construct toric gravitational instantons, the central task is to choose a rod structure $\mR$ such that all cone angles $\vartheta_j$ are equal to $2\pi$. A natural idea is to achieve this by varying the rod lengths. At first glance, this is a reasonable strategy, since the number of parameters matches the number of constraints, and it is not difficult to see that the cone angles $\vartheta_j$ depend continuously on the rod lengths. However, in general, we do not have an explicit description of the associated axisymmetric harmonic maps, and hence of the cone angles, making the situation more subtle.

First, it may happen that the cone angles $\vartheta_j$ do not vary at all when the rod lengths are changed. As a simple example, consider the double cover of the Calabi-Eguchi-Hanson metric branched along the totally geodesic $2$-sphere. This yields a gravitational instanton with cone angle $4\pi$ along the $2$-sphere. It is associated with a rod structure of asymptotic Type $\Ae$ with three rods, whose rod vectors are $(1,0)^{\top}$, $(1,1)^{\top}$, and $(0,1)^{\top}$. In this case, there is only one finite rod, and varying its length merely rescales the metric. The cone angle remains equal to $4\pi$ and hence does not change.

More generally, one has the following (see \cite{LiSun2025}, Remark 7.8):
\begin{proposition}
If $d(\mathfrak R) = 0$, then all cone angles $\vartheta_j$ remain constant under variations of the rod lengths.
\end{proposition}

Secondly, there are cases in which the cone angles vary in a favorable way. As an example, consider the Kerr metrics, which form an explicit family of smooth $\AF_\beta$ gravitational instantons $g_\beta$ on $\mathbb{R}^2 \times S^2$ for $\beta \in [0,1)$. The associated rod structure consists of three rods with directions $(1,0)^\top$, $(0,1)^\top$, and $(1,0)^\top$, where the vector $(0,1)^\top$ generates the screw rotation at infinity in the $\AF_\beta$ model.

When $\beta = 0$, one recovers the Schwarzschild metric, while as $\beta \to 1$, an explicit analysis shows that the length $l_\beta$ of the finite rod tends to infinity. Now fix $\beta \in (0,1)$. For any $\gamma \in (0,1)$, one can modify the smooth metric $g_\gamma$ to produce a conical $\AF_\beta$ gravitational instanton with varying rod length as follows.

First, replace the lattice $\Lambda$ by the lattice generated by $(1,0)^\top$ and $(0,\beta\gamma^{-1})^\top$. This produces a toric Ricci-flat metric with cone angle
\[
\vartheta_{\beta,\gamma} = 2\pi \beta \gamma^{-1}
\]
along a $2$-sphere. Next, rescale the metric by a factor $\beta^{-1}\gamma$. The resulting metric $g_{\beta,\gamma}$ has $\AF_\beta$ asymptotics, with associated rod length
\[
l_{\beta,\gamma} = l_\gamma \, \beta^{-1} \gamma.
\]
As $l_{\beta,\gamma} \to \infty$ (since $\beta$ is fixed and $l_\gamma\to\infty$ happens when $\gamma\to1$, $l_{\beta,\gamma}\to\infty$ happens precisely when $\gamma \to 1$), the cone angle $\vartheta_{\beta,\gamma}$ tends to $2\pi\beta < 2\pi$, while as $l_{\beta,\gamma} \to 0$ (i.e., $\gamma \to 0$), the cone angle $\vartheta_{\beta,\gamma}$ diverges to infinity. By continuity of the cone angle, there exists a critical value $\gamma_0$ such that $\vartheta_{\beta,\gamma_0} = 2\pi$, and the corresponding metric $g_{\beta,\gamma_0}$ is smooth. Of course, we already know this occurs when $\gamma = \beta$, but this discussion illustrates that, in favorable situations, one can tune the cone angles to $2\pi$ by varying the rod length.

More generally, in the Type $\mathrm{II}$ case, the results of Biquard--Gauduchon \cite{BG} provide explicit relations between cone angles and rod lengths, which in particular imply the uniqueness of the critical rod length (up to scaling).

\subsection{The gravitational instantons in \texorpdfstring{\cite{LiSun2025}}{LiSun2025}}

Fix $\beta\in (0, 1)$ and an integer $n\geq 1$. For $\bl=(l_1, \cdots, l_n)\in \R^n_+$,  we consider the rod structure  $\mathfrak{R}_{\bl}$ with

\begin{itemize}
\item turning points $\{z_j\}_{j=1}^{n+1}$, such that $z_1=0$ and $z_{j+1}-z_{j}=l_j$ for $j=1, \cdots, n$;
\item lattice $\Lambda=\mathbb Z^2\subset \mathbb R^2$;
    \item rod vectors $\mathfrak{v}_0=\mathfrak{v}_{n+1}=(1,0)^{\top}$ and for $j\neq0,n+1$, $\mathfrak{v}_{j}=(0,1)^{\top}$ when $j$ is odd, and $\mathfrak{v}_{j}=(-1,1)^{\top}$ when $j$ is even;
    \item asymptotic Type $\AFa$
\end{itemize}
 We let $(\Phi_{\bl}, \nu_{\bl})$ be the augmented harmonic map which is strongly tamed by  $\mR_{\bl}$, and let $g_{\bl}$ be the associated conical $\AFa$ gravitational instanton on $X_n$. Here $X_n$ is the oriented 4-manifold diffeomorphic to 
 \begin{equation}\label{eqn:definition of Xn}X_n:=\begin{cases}(S^2\times \R^2)\# k {\CP^2}\# k \overline{\CP^2},  & \text{if } n=2k+1;
\\ (S^2\times \R^2)\# k{\CP^2} \# (k+1)\overline{\CP^2}, & \text{if } n=2k+2.
\end{cases}
\end{equation}
They are simply-connected and non-spin. In particular, it is clear that they do not admit Type $\I$
 gravitational instantons. 
 
\begin{theorem}[\cite{LiSun2025}]\label{t:adjusting angle}
There exists an $\bl\in \R^n_+$ such that $g_{\bl}$ is smooth (so is an $\AFa$ gravitational instanton).
\end{theorem}
\begin{remark}
    It is not known from the proof whether the $\bl$ above is uniquely determined by $\beta$. 
\end{remark}
\begin{remark}
Notice that $d(\mR_l)=1$ when $n=1, 2$,  $\min (d(\mR_l), d(\overline{\mR}_l))>1$ when $n\geq 3$. 
By the classification result Theorem \ref{thm:equivalence of degree and type}, we know that when $n=1$ or $n=2$, the gravitational instantons constructed above coincide with the Kerr spaces and the Chen-Teo spaces, and when $n\geq 3$  the gravitational instantons constructed above is Type $\III$ for both orientations. 
\end{remark}

We give a brief outline of the main ideas in the proof here. Denote by $\vartheta_j(\bl)$  the cone angle along the 2-sphere associated to $\mI_j$. By construction, we know that $\vartheta_0=\vartheta_{n+1}=2\pi$.  We define the \emph{angle} map 
\begin{equation}
    \mathcal{A}:\mathbb{R}^n_+\to\mathbb{R}^n_+;(l_1,\ldots,l_n)\mapsto (\vartheta_1(\bl),\ldots,\vartheta_n(\bl)).
    \end{equation}
It suffices to show that the image of $\mathcal{A}$ contains the point $${\btheta}_0:=(2\pi,\ldots,2\pi).$$  Denote  
    $$\mathcal{A}_{0}:\mathbb{R}^n_+\to\mathbb{R}^n_+; (l_1,\ldots,l_n)\mapsto(e^{-l_1}(2\pi+1),\ldots,e^{-l_n}(2\pi+1)).$$
Theorem \ref{t:adjusting angle} follows from the following result and a simple degree argument. 
\begin{proposition} \label{p:boundary avoid}There are $0<\epsilon<N$ such that the boundary of the cube $[\epsilon,N]^n$  is contained in the set \    $$\mathcal S:=\{\bl=(l_1, \cdots, l_n)\in \R_n^+\mid t\mathcal{A}(\bl)+(1-t)\mathcal{A}_0(\bl)\neq \btheta_0, \ \ \forall \ \  t\in [0, 1]\}.$$
\end{proposition}
This is in turn a consequence of the following two statements 

\begin{itemize}
    \item There exists an $N>0$ such that if $l_j>N$ for some $j$, then $(l_1, \cdots, l_n)\in \mathcal S$.
    \item  There exits an $\epsilon>0$ such that if $l_j<2N$ for all $j$ and $l_j<\epsilon$ for some $j$, then $(l_1, \cdots, l_n)\in \mathcal S.$
\end{itemize}

As mentioned above, the main difficulty lies in the lack of explicit control over the cone angles in general. Nevertheless, for the purposes of the above statements, it suffices to obtain limiting information on the cone angles when the rod structure degenerates, i.e., when one of the rod lengths tends to zero or infinity. This is achieved via a bubbling analysis of harmonic maps, carried out in \cite[Section~5]{LiSun2025}, by comparison with a family of \emph{background} maps constructed by gluing suitably rescaled model maps through an induction on scale.

It is worth noting that, although the rod structures under consideration are of asymptotic Type $\AF_\beta$, the bubbling limits may involve rod structures of asymptotic Type $\Ae$, corresponding to non-collapsed limits of the associated toric gravitational instantons. On the other hand, during the bubbling process, the lattice $\Lambda$ appearing in the definition of the rod structure may degenerate. As a result, the limiting harmonic maps need not correspond to genuine rod structures, and hence may fail to give rise to toric gravitational instantons. This phenomenon reflects the presence of \emph{collapsing} in the degeneration.

While the bubbling analysis does not provide a complete description of the cone angles, it yields two key estimates that are sufficient to establish the two statements above.

First, suppose that the maximal rod length, say $l_j$, tends to infinity. By rescaling the rod structure by $l_j^{-1}$, one obtains a degeneration of the lattice $\Lambda$. For the rod structures under consideration, one can show that the corresponding cone angle $\vartheta_j$ converges to either $2\pi\beta$ or $2\pi(1-\beta)$, depending on the parity of $j$. In particular, it is always strictly less than $2\pi$. This also explains the requirement that $\beta \in (0,1)$ in our construction. The Kerr metrics discussed in Section~\ref{ss:cone angles and rod lengths} provide a concrete example of this phenomenon.

Second, if a bubble limit has a rod structure $\mR'$ of asymptotic Type $\Ae$, then one obtains an \emph{angle comparison} estimate: if the rod structure has $k$ turning points, then
\[
\vartheta_1 \ge \vartheta_k, \qquad \vartheta_{k-1} \ge \vartheta_0.
\]
In particular, if $\vartheta_k = 2\pi$, then $\vartheta_1 \ge 2\pi$. This comparison ultimately follows from the Bishop--Gromov volume comparison principle.

\begin{remark}
    From our proof it is not clear whether $\bl$ depends uniquely or continuously on $\beta$. There is also no explicit understanding of the geometry of $g_{\bl}$. In \cite{LiSun2025}, it is shown that the limit of $g_{\bl}$ as $\beta\rightarrow 0$ can nevertheless be understood. More precisely, 
as $\beta\rightarrow 0$, in the pointed Gromov-Hausdorff sense, without rescalings, $(X_{n}, g_\beta)$ splits into the union of 
   \begin{itemize}
       \item A  Taub-NUT metric, an anti-Taub-Bolt metric and $k-1$ Schwarzschild metrics (when $n=2k\geq 2$);
       \item An anti-Taub-Bolt metric, a  Taub-Bolt metric and $k-1$ Schwarzschild metrics (when $n=2k+1\geq3$);
   \end{itemize}

    When $n=1$ or $n=2$ it is also possible to see this directly using the explicit expression of the Kerr metrics and the Chen-Teo metrics. 
\end{remark}

\subsection{The gravitational instantons of  Khuri-Weinstein-Yamada \texorpdfstring{\cite{KWY}}{KWY}} \label{ss:KWY}
In \cite{KWY}, another class of toric Type $\III$ gravitational instantons was constructed earlier via an explicit ansatz, which in turn depends on a simplification of the axisymmetric harmonic map formalism. There is a very similar construction in the Lorenztian setting \cite{KorotkinNicolai1994PeriodicAnalog}. There is an associated generalized rod structure consisting of infinitely many turning points arranged in a periodic configuration. Equivalently, one may view this as a rod structure in which $\mathbb R$ is replaced by $S^1$.

The construction is based on the observation that when a map $\Phi: \mathbb H^\circ\rightarrow SL(2;\R)/SO(2)$ is diagonalizable, i.e., of the form
$$
\Phi = \begin{pmatrix}
e^{u_1} & 0 \\
0 & e^{u_2}
\end{pmatrix},
$$
then $\Phi$ is an axisymmetric harmonic map if and only if $u_1$ and $u_2$ are axisymmetric harmonic functions. Correspondingly the image of $\Phi$ is a \emph{geodesic} in $SL(2; \R)/SO(2)$. In the special case when $u_1 = -\log \rho$ and $u_2 = \log \rho$, this yields the flat product $\mathbb R^3 \times S^1$. 
	
To construct nontrivial toric gravitational instantons, \cite{KRWY} considers axisymmetric harmonic functions that are periodic in the $z$-direction. Equivalently, one works with axisymmetric harmonic functions on the product $Q = S^1 \times \mathbb{R}^2$.
More precisely, divide $S^1$ into two segments of equal length, denoted by $I_1$ and $I_2$. These can be naturally viewed as subsets of $Q$. One can then construct Green's functions $u_\alpha$ ($\alpha = 1,2$) on $Q$ such that $u_\alpha$ is asymptotic to $-\log \rho$ near points in the interior $I_\alpha^\circ$ and also at infinity.
Using $u_1$ and $u_2$ to define $\Phi$, one can then solve for $\nu$ according to \eqref{e:eqn2.33}, which is uniquely determined up to an additive constant.
	
	This yields a toric Ricci-flat metric on 
$$
M^\circ = (\mathbb H^\circ / \mathbb Z) \times (\mathbb R^2 / \Lambda),
$$
where the lattice $\Lambda$ is generated by $(1,0)^\top$ and $(0,1)^\top$. As in the previous discussion, one obtains a toric gravitational instanton $(M,g)$ with, in general, cone singularities along two 2-spheres. The cone angles along the two rods are given by
$$
2\pi \lim_{\rho\rightarrow 0}e^{-\log \rho - \nu + \frac{1}{2}u_i}.
$$
By adjusting $u_i$ by suitable additive constants, one can arrange for the metric to extend smoothly.
	
	 The underlying manifold $M$ is diffeomorphic to $S^4\setminus T^2$. It has fundamental group isomorphic to $\dZ$. The universal cover $\widetilde M$ is an infinite connected sum of $S^2\times S^2$, which admits a complete Ricci-flat metric but it has infinite $b_2$ and infinite curvature energy. By taking finite covers of $M$ we can also get a gravitational instanton on the connected sum of an arbitrary number of $S^2\times S^2$'s. 
	 
The metric $g$ also has an interesting asymptotic, namely, it is asymptotic to the special Kasner metric  
$$dr^2+r^{-\frac{2}{3}}dx_1^2+r^{\frac{4}{3}}dx_2^2+r^{\frac{4}{3}}dx_3^2.$$
It has quadratic volume growth and exactly quadratic curvature decay.	

\section{Open questions}\label{s:open problems}
In this section, we list some open questions related to gravitational instantons. 

\subsection{Toric gravitational instantons}
Following the construction of  toric \(\AFa\) gravitational instantons in Theorem~\ref{t:adjusting angle}, it is a natural problem to develop a systematic understanding of general toric gravitational instantons.

\begin{problem}
Classify toric gravitational instantons.
\end{problem}

In particular, for other types of asymptotic behavior, one may ask:

\begin{question}
Are there Type $\III$ toric gravitational instantons (for both orientations)  with $\ALE$ or $\ALF$ asymptotics?
\end{question}

 It is well known that toric gravitational instantons give rise to \emph{integrable systems}, and one may hope to apply the inverse scattering methods to generate explicit solutions. This approach has already proved successful in the discovery of the Chen-Teo gravitational instantons. Note that the new gravitational instantons  in Theorem~\ref{t:adjusting angle} were constructed by a topological degree argument. This leads to:

\begin{question}
Do toric gravitational instantons always admit an explicit description?
\end{question}

A first step in this direction would be to recover the gravitational instantons constructed in Theorem~\ref{t:adjusting angle} by means of integrable systems, especially for those rod structures with low degree. There is the very recent work \cite{Teo2026NewAsymptoticallyFlatGravitationalInstanton} that provides an explicit construction for the $\AFa$ gravitational instantons constructed in \cite{LiSun2025} on $X_3$.

\subsection{Asymptotic classification}

\begin{problem}\label{prob1}
    Classify the asymptotic structure of gravitational instantons.
\end{problem}

In the non-collapsed setting, a gravitational instanton is known to be $\ALE$ by \cite{BKN}, so the main difficulty lies in the collapsed setting. As explained in this article, the asymptotic classification in the collapsed case has already been completed in the hyperk\"ahler and Hermitian settings. The Type $\II$ case can always be reduced to the Hermitian case, up to passing to a double cover. The Type $\I$ case can be reduced to the hyperk\"ahler case, up to passing to a finite cover, provided the fundamental group is finite. This naturally leads to the following question:

\begin{question}\label{q:Type I infinite fund}
Does there exist a non-flat Type $\I$ gravitational instanton with infinite fundamental group?
\end{question}

Observe that the known examples of Type $\III$ gravitational instantons, discussed in Section \ref{s:Type III}, are all asymptotic to Type $\I$ or Type $\II$ model ends (up to orientation reversal).  There are  many Ricci-flat ends with symmetry which are Type $\III$ (for both orientations), for example, the general Kasner family of ends given in \eqref{eqn:Kasner}. One may ask 

\begin{question}
    Is there a gravitational instanton which is asymptotic to a Type $\III$ model end? 
\end{question}

In general, one of the main difficulties in addressing Problem~\ref{prob1} is the lack of a geometric understanding of asymptotic cones. 

\begin{conj}\label{q:conicity}
An asymptotic cone of a gravitational instanton is always a metric cone.
\end{conj}

\begin{conj}\label{q:uniqueness of asymptotic cone}
A gravitational instanton has a unique asymptotic cone.
\end{conj}

There is also a topological finiteness question
\begin{conj}
A gravitational instanton must be of finite topological type, i.e.,  it  can compactified into a  smooth manifold with boundary.
\end{conj}

All of these questions are known to have affirmative answers for general complete Riemannian manifolds whose curvature decays \emph{faster} than quadratically, for example when \(|\Rm_g| = O(r^{-2-\epsilon})\) for some \(\epsilon > 0\); see \cite{Kasue, MashikoNaganoOtsuka2005AsymptoticCones, PT}. 
If, in addition, the end is assumed to be simply connected, then the asymptotic cone must be flat: it is \(\R^m\) when \(m = \dim M \neq 4\), and it is either \(\R^4\) or \(\R^3\) when \(m = 4\). This follows from the work of Petrunin-Tuschmann \cite{PT}, together with the proof of the Petrunin-Tuschmann conjecture in \cite{LS2024}, which excludes the two dimensional  half-plane \(\mathbb H\) when $m=4$.
There is also a partial result toward Problem~\ref{prob1} in \cite{ChenLi}, obtained under an additional asymptotic holonomy assumption and curvature decay.

\

The diversity of asymptotic models of gravitational instantons also generates various analytic questions, for example, Fredholm theory for natural elliptic operators and Hodge theory on gravitational instantons. Especially interesting is the case where the asymptotics is not standard, for example, in the $\ALG^*, \ALH^*, \AFa(\beta\notin\mathbb Q)$ cases. 

\begin{problem}
    Develop the elliptic theory for general asymptotics of gravitational instantons. 
\end{problem}

This will likely involve the techniques of microlocal analysis. See, for example, \cite{HHM04,mazzeo2025tianyaumetricsfredholmtheory} and references therein.

\subsection{Rigidity with fixed asymptotics}

Given a prescribed asymptotic structure at infinity, there are several longstanding rigidity questions. In the \(\ALE\) case, one has the conjecture of Bando-Kasue-Nakajima:

\begin{conj}[\cite{BKN}]
   An $\ALE$ gravitational instanton must be of Type $\I$, up to orientation reversal.
\end{conj}

There are some partial results on this question, see for example \cite{Nakajima1990SelfDualityALE,LockViaclovsky2016QuotientSingularities, Li1}. To understand $\ALE$ gravitational instantons of  Type $\II$ requires an extension of Theorem \ref{thm:Hermitian B} in the case when $\Gamma$ is not in $SU(2)$. 

\

In the \(\AF\) case, there is a \emph{Riemannian black hole uniqueness conjecture}, which has been discussed by many authors; see, for example, Gibbons \cite{Gibbons}, Lapedes \cite{Lapedes}, Yau (\cite{Yau1982problem}, Problem~116), and Aksteiner-Andersson-Dahl-Nilsson-Simon \cite{AADNS}. In general relativity, the black hole uniqueness conjecture asserts that the only asymptotically flat, stationary vacuum Einstein spacetime is the Kerr spacetime. 
The naive Riemannian analogue-namely, the uniqueness of Kerr metrics among \(\AF\) gravitational instantons, was disproved by the Chen-Teo metrics, and more generally by the metrics constructed in Theorem~\ref{t:adjusting angle}. A modified version of the Riemannian black hole uniqueness conjecture were proposed in \cite{LiSun2025}. 
\begin{BHUconjI}\label{BHUconjI}
    The Schwarzschild metric is, up to isometry, the unique non-flat $\AF_0$ gravitational instanton. 
\end{BHUconjI}
         
Evidence for this is provided by the uniqueness result of Mars-Simon \cite{MarsSimon} (with a mild extension in \cite{LiSun2025}), which implies that the above conjecture holds in the presence of an isometric semi-free \(S^1\)-action compatible with the \(\AF_0\) asymptotics. The proof relies on a conformal argument due to Bunting-Masood-ul-Alam \cite{BuntingMasood1987}, which reduces the problem to an application of the positive mass theorem.
\

There is a natural extension of the above conjecture to the \(\ALF\) case. Again, supporting evidence is available under the assumption of an \(S^1\)-symmetry.

\begin{conj}
The multi-Taub-NUT metrics and the Taub-Bolt metric are the only \(\ALF\)-\(A_k\) $(k\geq 0)$ gravitational instantons, up to isometry, scaling, and orientation reversal.
\end{conj}

\subsection{Topological obstructions and rigidity}
Besides the general topological constraints arising from nonnegative Ricci curvature, an obstruction specific to gravitational instantons is given by the Hitchin-Thorpe inequality. Recall that for a compact oriented Einstein 4-manifold \((M, g)\), we have
\begin{equation}
2\chi(M) + 3\tau(M) \geq 0.
\end{equation}
where $\tau(M)$ is the signature of $M$, and the equality holds if and only if $g$ is Ricci-flat and Type $\I$. This follows from a combination of the Gauss-Bonnet-Chern theorem and the signature formula
\begin{equation}
    \tau(M)=\frac{1}{12\pi^2}\int_M(|\W^+|^2-|\W^-|^2)\dvol_g.
\end{equation}
For gravitational instantons with given asymptotics, an expected general form of the Hitchin-Thorpe inequality is
\begin{equation}
2\bigl(\chi(M) - \delta(g)\bigr) + 3\bigl(\tau(M) - \eta(g)\bigr)\geq0.
\end{equation}
Here the correction term \(\delta(g)\) to \(\chi(M)\) is given in Lemma~\ref{lem:Gauss-Bonnet}, while the correction term \(\eta(g)\) to \(\tau(M)\) involves the limiting eta invariant; see, for example, \cite{Nakajima1990SelfDualityALE,DaiWei2007HitchinThorpe,LockViaclovsky2016QuotientSingularities, Chen-Li}. In particular, for the standard \(\ALE\), \(\ALF\), \(\ALG\), and \(\ALH\) asymptotics, these correction terms can be computed explicitly.

As a consequence, if \(X\) admits a hyperk\"ahler gravitational instanton with \(\ALs\) (\(\sharp = E, F, G, H\)) asymptotics, then any gravitational instanton on \(X\) with the same \(\ALs\) asymptotics must be hyperk\"ahler, up to orientation reversal. This also yields certain nonexistence results under prescribed asymptotics.

Motivated by this, it is a natural to ask similar rigidity questions. 

\begin{conj}\label{c:TB uniqueness with topology}
The Taub-Bolt metric is the unique \(\ALF\)-\(A_k\) gravitational instanton on the total space of \(\mathcal{O}(1)\) over $\mathbb{CP}^1$, up to isometry and scaling.
\end{conj}

\begin{conj}\label{BHUconjII}
For each \(\beta \in (0,1)\) and \(n \geq 1\), the metric constructed in Theorem~\ref{t:adjusting angle} is the unique \(\AFa\) gravitational instanton on \(X_n\), up to isometry.
\end{conj}

There are some partial results on these questions assuming an \(S^1\)-symmetry;  see \cite{AADNS}. 

   \

For manifolds with simple topology, one expects a general rigidity

\begin{conj}
    A gravitational instanton on a topological $\R^4$ must be either the flat metric or the Taub-NUT metric, up to isometry, scaling and orientation reversal. 
\end{conj}
 This is known to hold when the curvature decays faster than quadratic rate, by \cite{Chen-Li, LS2024}.

\begin{problem}
    Classify gravitational instantons with $b_1>0$. 
\end{problem}

This is related to Question \ref{q:Type I infinite fund}. The only known examples of non-flat gravitational instantons with $b_1>0$ are the ones in \cite{KRWY}.

\

Recall that for compact Einstein 4-manifolds, there are obstructions arising from smooth topology, for example through the Seiberg-Witten theory \cite{LeBrun1996,LeBrun2001Ricci,LeBrun1995Mostow,LeBrun1998Weyl,LeBrun2003DifferentialTopology,IshidaLeBrun2002Spin}. One may ask whether analogous phenomena occur in the noncompact setting.

\begin{problem}
Find smooth obstructions to the existence of gravitational instantons.
\end{problem}

\subsection{Mass and renormalized volume}

Since a gravitational instanton necessarily has vanishing scalar curvature, results concerning nonnegative scalar curvature impose constraints on gravitational instantons as well. The classical positive mass theorem applies to asymptotically Euclidean (\(\Ae\), i.e., \(\ALE\) with trivial \(\Gamma\)) manifolds \((M,g)\) with nonnegative scalar curvature. Since gravitational instantons exhibit a variety of interesting asymptotic geometries, it is natural to seek a notion of \emph{mass} that can be defined and computed in terms of the asymptotic structure.

\begin{problem}
    Formulate and prove a positive mass theorem for Riemannian metrics with the asymptotics of gravitational instantons. 
\end{problem}

There has already been a lot of progress in this direction; see, for example, \cite{Nakajima1990SelfDualityALE, Dahl1997, Dai2004,  Minerbe2009mass,  HeinLeBrun2016, LiuShiZhu2021,ChenLiuShiZhu2021, DaiSun2023, khuri2025masslowerboundsasymptotically,alaee2026comparison}. It is particularly interesting to understand the condition of non-negative Ricci curvature, as opposed to non-negative scalar curvature. 

In the classical setting, zero mass is achieved only by flat metrics. For the asymptotics of gravitational instantons, typically  flat models cannot be realized, and thus one expects the existence of a definite mass gap. This is analogous to the Penrose-type inequality in general relativity. In the Riemannian setting, the area of the horizon can be replaced by the area of certain minimal submanifolds.

\begin{problem}
  Formulate and prove a Penrose-type inequality characterizing  the Schwarzschild metric and the Kerr metrics.
\end{problem}
There are related results, see, for example \cite{HeinLeBrun2016, khuri2025masslowerboundsasymptotically}. 

\

In a different direction, recall that for $\ALE$ gravitational instantons, Biquard--Hein \cite{BiquardHein2023RenormalizedVolume} introduced a notion of \emph{renormalized volume} $\mathcal V$, using a canonical foliation of the end by constant-mean-curvature hypersurfaces (see also \cite{wang2025curvatureinfinityasymptoticallyflat} for a different description). This renormalized volume is a global geometric invariant of an $\ALE$ gravitational instanton. It is similar to the notion of renormalized volume for Poincar\'e-Einstein manifolds. 

\begin{question}
    Is there an analogous notion of renormalized volume for gravitational instantons with other asymptotics?
\end{question}

\subsection{Stability}
Gravitational instantons are critical points of the Einstein--Hilbert functional. For reasons arising from physics and from Ricci flow, it is important to investigate their stability.

\begin{problem}
   Investigate the stability of gravitational instantons.
\end{problem}

Note that, a priori, there may be several different notions of stability. There are many results on this problem; see, for example, \cite{DaiWangWei2005ParallelSpinors,Sesum2006LinearDynamical,DeruelleKroencke2021ALEStability,DeruelleOzuch2023ALEInstability,BiquardOzuch2023Instability,AnderssonAraneda2026,HaslhoferMueller2014RicciFlat,KimOzuch2025ALF}. In particular, it is shown by Biquard-Ozuch \cite{BiquardOzuch2023Instability} that Hermitian gravitational instantons are all linearly unstable. A natural question is whether the Type $\III$ gravitational instantons discussed in Section \ref{s:Type III} are stable.

\subsection{Gauge theory}
There are deep  connections between gravitational instantons and gauge theory, which we did not discuss in detail in this survey. A general problem is

\begin{problem}
Realize gravitational instantons as moduli spaces arising in gauge theory.
\end{problem}

For hyperk\"ahler gravitational instantons, there are explicit conjectures connecting to moduli spaces in gauge theory. On the one hand, there is a  modularity conjecture (\cite{Boalch2015OpenQuestions}) which predicts that any  hyperk\"ahler gravitational instanton with quadratic volume growth arises as some Hitchin moduli spaces (of Higgs bundles), see \cite{MazzeoSwobodaWeissWitt2016,MazzeoSwobodaWeissWitt2019,ChenLi,fredrickson2026alggravitationalinstantonshitchin} for progress in this direction. On the other hand,  hyperk\"ahler gravitational instantons are also expected to arise as suitable monopole moduli spaces, see \cite{CherkisKapustin2002, CherkisWard2012, HamanakaKannoMuranaka2014, MaldonadoWard2014, Cross2015, Foscolo2016, Foscolo2017Gluing, ThomasHarris2026} for conjectures and progress. 

\

It is also a natural question to study gauge theory over gravitational instantons. 

\begin{problem}
    Classify finite energy ASD (anti-self-dual) instantons over a given gravitational instanton. 
\end{problem}

For the flat $\R^4$ this is the famous ADHM construction. For $\ALE$ and $\ALF$ hyperk\"ahler gravitational instantons, there are extensive studies on this topic, see, for example \cite{KronheimerNakajima1990, Cherkis2011, CherkisLarrainHubachStern2021, CherkisLarrainHubachStern2024, CherkisLarrainHubachStern2026} and the references therein.

\subsection{Submanifold geometry}
Dually, it is natural to investigate special submanifolds, for example, complete or compact minimal submanifolds, submanifolds of constant mean curvature, etc,  in a gravitational instanton.

\begin{problem}
    Construct and classify special submanifolds in gravitational instantons.
\end{problem}

There is a long line of research on understanding stable minimal surfaces in Ricci-flat manifolds. For example, this is related to a question of Yau (\cite[Question 64]{Yau1993OpenProblems}), see, for example, \cite{Micallef1984,MicallefWolfson2006,foscolo2024unstableminimalspheresdegree1}.

\subsection{Moduli and bubbling}

Even though hyperk\"ahler gravitational instantons have been classified, it remains to understand the moduli compactification of these spaces, similar to studying the case of compact hyperk\"ahler metrics on the K3 manifolds in \cite{SZ}. See \cite{LinTakahashi2023ALHStar,LinSouZhu} for progress in this direction. 

\begin{problem}\label{p:moduli of GI}
    Study moduli compactification of hyperk\"ahler gravitational instantons and the connections to algebraic geometry. 
\end{problem}

Due to noncompactness, taking limits of gravitational instantons would inevitably involve choosing base points and scales. These will be important in understand general bubbling analysis for singularity formations. On the other hand, as is discussed in \cite{SZ}, it would be an interesting question to understand the multi-scale moduli of hyperk\"ahler metrics, either compact, or non-compact, encoding the bubbling information, and relate them to period domains (in the spirit of \cite{OdakaOshima2021CollapsingK3}).

\subsection{Relation with compact Einstein metrics}
As explained in Section \ref{ss:GI and K3}, hyperk\"ahler gravitational instantons can  arise from singularity formation of Ricci-flat metrics on the K3 manifold. More generally, Type $\I$ gravitational instantons may also come out of singularity formation of general compact K\"ahler-Einstein metrics in complex dimension 2. In the non-collapsing situation, which is automatic when the Ricci curvature is positive,  this is well-understood, see, for example,  \cite{OdakaSpottiSun2016DelPezzo, BiquardRollin2013Smoothing,deborbon2023modelsbubblinglogkahlereinstein}.  In the negatively curved case, collapsing may occur and one may see non-$\ALE$ gravitational instantons, see, for example, \cite{MR4680200,fu2025continuouscuspclosingprocess}. 

\begin{problem}
Understand the gravitational instantons bubbling off from degeneration of K\"
ahler-Einstein metrics with negative Ricci curvature.
\end{problem}

This would give extra geometric information on the KSBA moduli space for algebraic surfaces of general type.

\

In the more general Riemannian setting, a natural question is 

\begin{question}
    Can Type $\II$ or Type $\III$ gravitational instantons arise from singularity formation on compact Einstein metrics?
\end{question}

Note that currently we do not have many examples of compact Einstein metrics in 4 dimensions. In particular, all known examples with non-negative Einstein constants are of either Type $\I$ or Type $\II$ for a choice of orientation. Note that the notion of Type can be defined for general Einstein metrics.   Naturally we can ask 

\begin{question}
  Is there a compact Einstein 4-manifold with non-negative Einstein constant which is of Type $\III$ for both orientations? 
\end{question}

In the Ricci-flat case, this is related to a longstanding question of Page:

\begin{question}\label{q:page}
    Is there a compact simply-connected Ricci-flat 4-manifold which is not diffeomorphic to the K3 manifold?
\end{question}

Following a proposal of Gibbons-Pope and Page, there are extensive results towards understanding a potential variant of the Kummer construction, where one reverses the orientation of some of the Calabi-Eguchi-Hanson spaces, to construct a Ricci-flat manifold not diffeomorphic to the K3 manifold, see \cite{brendle2016gluing, ozuch2022higher}. These results are all negative, suggesting the difficulty of Question \ref{q:page}.

\subsection{Infinite energy}

In general, for a complete oriented Ricci-flat 4-manifold,  the second Betti number $b_2$ could be infinity. There are two classes of known examples, both given by explicit constructions. 

The first class of examples is given by Anderson-Kronheimer-LeBrun \cite{AKL}, which applies the Gibbons-Hawking ansatz to a positive harmonic function on $\R^3$ with a suitably chosen infinite set of singularities. In particular, these examples are hyperk\"ahler and they are often referred to as Type $A_\infty$. Their geometry is much more complicated than gravitational instantons.  In particular, it is known \cite{Hattori} that the asymptotic cones may not be unique and may not be metric cones. 

The second class of examples is given by the universal covers of the gravitational instantons of \cite{KRWY}, as discussed in Section \ref{ss:KWY}. 

\

A general conjecture is 

\begin{conj}[\cite{SZ}]
   A complete oriented Ricci-flat 4-manifold $(M, g)$ is a gravitational instanton if and only if $M$ has finite topological type. 
\end{conj}

This is related to 

\begin{question}
    Can an exotic $\R^4$ admit a complete Ricci-flat metric? 
\end{question}

It is also unknown whether an exotic $\R^4$ can admit a complete Riemannian metric with non-negative Ricci curvature.

\bibliographystyle{alpha}
\bibliography{ref.bib}

\end{document}